\documentclass[12pt]{article}

\usepackage{bbm}
\usepackage{fullpage}
\usepackage{amsmath,amsthm,amssymb}
\usepackage{mathtools}
\usepackage{mathrsfs}
\usepackage{graphicx}
\usepackage{float}
\usepackage{tikz}
\usepackage{tikz-layers}
\usetikzlibrary{math}
\usepackage{caption}
\usepackage{subcaption}
\usepackage{enumitem}
\usepackage{mathabx}

\usepackage{hyperref}

\newtheorem{definition}{Definition}[section]

\newtheorem{theorem}[definition]{Theorem}
\newtheorem{lemma}[definition]{Lemma}
\newtheorem{proposition}[definition]{Proposition}
\newtheorem{claim}[definition]{Claim}

\newtheorem{question}[definition]{Question}
\newtheorem{corollary}[definition]{Corollary}
\newtheorem{remark}[definition]{Remark}

\def \mH {\mathcal{H}}
\def \ex {\mathrm{ex}}
\def \sm {\setminus}
\def \ce {\coloneqq}

\def \pa {\bigg\lfloor \frac{n}{3} \bigg\rfloor}
\def \pb {\left\lfloor \frac{n+1}{3} \right\rfloor}
\def \pc {\left\lfloor \frac{n+2}{3} \right\rfloor}
\def \tri {\triangle}
\def \F {\mathcal{F}}

\def \ffp {F_5^{t}}
\def \ftp {F_2^{t}}

\def \S {\mathcal{S}}
\def \U {\mathcal{U}}

\def \supforOne {\circ}
\def \supforTwo {\tri}

\newcommand \fHo[1][n] {\S_t^\supforOne(#1)}
\newcommand \Ho[1][n] {S_t^\supforOne(#1)}

\newcommand \sHo[1][n] {s_t^\supforOne(#1)}

\newcommand \fHop[1][n] {\S_{t,\pi}^\supforOne(#1)}
\newcommand \Hop[1][n] {S_{t,\pi}^\supforOne(#1)}

\newcommand \Htp[1][n] {S_{t,\varpi}^\supforTwo(#1)}

\newcommand \sHop[1][n] {s_{t,\pi}^\supforOne(#1)}
\newcommand \sHtp[1][n] {s_{t,\varpi}^\supforTwo(#1)}

\renewcommand{\le}{\leqslant}
\renewcommand{\ge}{\geqslant}
\renewcommand{\leq}{\leqslant}
\renewcommand{\geq}{\geqslant}

\def \eps {\varepsilon}
\def \es {\varnothing}
\renewcommand \b[2] {\binom{#1}{#2}}
\def \bdm {\begin{displaymath}}
\def \edm {\end{displaymath}}

\usepackage{ifthen}
\usepackage{pgf}
\usepackage{tikz}
\usetikzlibrary{graphs}
\usetikzlibrary{positioning}
\usetikzlibrary{calc}

\pgfdeclarelayer{bg}
\pgfdeclarelayer{fg}
\pgfsetlayers{bg,main,fg}

\def\rotateclockwise#1{
	\newdimen\xrw
	\pgfextractx{\xrw}{#1}
	\newdimen\yrw
	\pgfextracty{\yrw}{#1}
	\pgfpoint{\yrw}{-\xrw}
}

\def\rotatecounterclockwise#1{
	\newdimen\xrcw
	\pgfextractx{\xrcw}{#1}
	\newdimen\yrcw
	\pgfextracty{\yrcw}{#1}
	\pgfpoint{-\yrcw}{\xrcw}
}

\def\outsidespacerpgfclockwise#1#2#3{
	\pgfpointscale{#3}{
		\rotateclockwise{
			\pgfpointnormalised{
				\pgfpointdiff{#1}{#2}}}}
}

\def\outsidespacerpgfcounterclockwise#1#2#3{
	\pgfpointscale{#3}{
		\rotatecounterclockwise{
			\pgfpointnormalised{
				\pgfpointdiff{#1}{#2}}}}
}

\def\outsidepgfclockwise#1#2#3{
	\pgfpointadd{#2}{\outsidespacerpgfclockwise{#1}{#2}{#3}}
}

\def\outsidepgfcounterclockwise#1#2#3{
	\pgfpointadd{#2}{\outsidespacerpgfcounterclockwise{#1}{#2}{#3}}
}

\def\outside#1#2#3{
	($ (#2) ! #3 ! -90 : (#1) $)
}

\def\cornerpgf#1#2#3#4{
	\pgfextra{
		\pgfmathanglebetweenpoints{#2}{\outsidepgfcounterclockwise{#1}{#2}{#4}}
		\let\anglea\pgfmathresult
		\let\startangle\pgfmathresult
		
		\pgfmathanglebetweenpoints{#2}{\outsidepgfclockwise{#3}{#2}{#4}}
		\pgfmathparse{\pgfmathresult - \anglea}
		\pgfmathroundto{\pgfmathresult}
		\let\arcangle\pgfmathresult
		\ifthenelse{180=\arcangle \or 180<\arcangle}{
			\pgfmathparse{-360 + \arcangle}}{
			\pgfmathparse{\arcangle}}
		\let\deltaangle\pgfmathresult
		
		\newdimen\x
		\pgfextractx{\x}{\outsidepgfcounterclockwise{#1}{#2}{#4}}
		\newdimen\y
		\pgfextracty{\y}{\outsidepgfcounterclockwise{#1}{#2}{#4}}
	}
	-- (\x,\y) arc [start angle=\startangle, delta angle=\deltaangle, radius=#4]
}

\def\corner#1#2#3#4{
	\cornerpgf{\pgfpointanchor{#1}{center}}{\pgfpointanchor{#2}{center}}{\pgfpointanchor{#3}{center}}{#4}
}

\def\hedgeiii#1#2#3#4{
	\outside{#1}{#2}{#4} \corner{#1}{#2}{#3}{#4} \corner{#2}{#3}{#1}{#4} \corner{#3}{#1}{#2}{#4} -- cycle
}

\def\hedgem#1#2#3#4{
	\outside{#1}{#2}{#4}
	\pgfextra{
		\def\hgnodea{#1}
		\def\hgnodeb{#2}
	}
	foreach \c in {#3} {
		\corner{\hgnodea}{\hgnodeb}{\c}{#4}
		\pgfextra{
			\global\let\hgnodea\hgnodeb
			\global\let\hgnodeb\c
		}
	}
	\corner{\hgnodea}{\hgnodeb}{#1}{#4}
	\corner{\hgnodeb}{#1}{#2}{#4}
	-- cycle
}

\def\hgrotate#1{
	\newdimen\x
	\pgfextractx{\x}{#1}
	\newdimen\y
	\pgfextracty{\y}{#1}
	\pgfpoint{-\y}{\x}
}

\def\hgperpr#1#2#3{
	\pgfpointscale{#3}{
		\hgrotate{
			\pgfpointnormalised{
				\pgfpointdiff{#1}{#2}}}}
}

\def\hgaddeperpr#1#2#3{
	\pgfpointadd{#2}{\hgperpr{#1}{#2}{#3}}
}

\def\hgaddsperpr#1#2#3{
	\pgfpointadd{\hgperpr{#1}{#2}{#3}}{#1}
}

\def\hgcorner#1#2#3#4{
	\pgflineto{\hgaddeperpr{#1}{#2}{#4}}
	\pgfmathanglebetweenpoints{#1}{#2}\let\anga\pgfmathresult
	\pgfpatharcto{#4}{#4}{90 + \anga}{0}{0}{\hgaddsperpr{#2}{#3}{#4}}
}

\title{Non-degenerate Hypergraphs with Exponentially \\Many Extremal Constructions}

\author{
J\'ozsef Balogh\thanks{Department of Mathematics, University of Illinois at Urbana-Champaign, Urbana, Illinois 61801, USA. E-mail: \texttt{jobal@illinois.edu}. Research is partially supported by NSF Grant DMS-1764123 and RTG DMS-1937241, Arnold O. Beckman Research Award (UIUC Campus Research Board RB 22000), the Langan Scholar Fund (UIUC), and the Simons Fellowship.}
\and 
Felix Christian Clemen\thanks{Department of Mathematics, Karlsruhe Institute of Technology, 76131 Karlsruhe, Germany, Email: \texttt{felix.clemen@kit.edu}. Research was partially performed while the second author was at the University of Illinois at Urbana-Champaign.}
\and 
Haoran Luo\thanks{Department of Mathematics, University of Illinois at Urbana-Champaign, Urbana, Illinois 61801, USA. E-mail: \texttt{haoranl8@illinois.edu}. Research is partially supported by UIUC Campus Research Board RB 22000.}
}

\date{}

\begin{document}
\maketitle{}

\begin{abstract}
For every integer $t \ge 0$, denote by $\ffp$ the hypergraph on vertex set $\{1,2,\ldots, 5+t\}$ with hyperedges $\{123,124\} \cup \{34k : 5 \le k \le 5+t\}$. We determine $\ex(n,\ffp)$ for every $t\ge 0$ and sufficiently large $n$ and characterize the extremal $\ffp$-free hypergraphs. In parti\-cular, if $n$ satisfies certain divisibility conditions, then the extremal $\ffp$-free hypergraphs are exactly the balanced complete tripartite hypergraphs with additional hyperedges inside each of the three parts $(V_1,V_2,V_3)$ in the partition; each part $V_i$ spans a $(|V_i|,3,2,t)$-design. This gene\-ralizes earlier work of Frankl and F\"uredi on the Tur\'an number of $F_5:=F_5^0$.

Our results extend a theory of Erd\H os and Simonovits about the extremal constructions for certain fixed graphs.
In particular, the hypergraphs $F_5^{6t}$, for $t\geq 1$,
are the first examples of hypergraphs with exponentially many extremal constructions and positive Tur\'an density. 
\end{abstract}

\section{Introduction} \label{sec::Int}
For a family of $r$-uniform hypergraphs (\emph{$r$-graphs}) $\F$, a hypergraph $H$ is \emph{$\F$-free} if $H$ contains no copy of any $F \in \F$ as a subhypergraph. The \emph{Tur\'an number} $\ex(n,\F)$ is the maxi\-mum possible number of hyperedges in an $n$-vertex $\F$-free $r$-graph. Let $\pi(\F) \ce \lim_{n\to \infty} \ex(n,\F) / \b{n}{r}$ be the \emph{Tur\'an density} of $\F$. 
A family $\F$ is \emph{non-degenerate} if $\pi(\F) >0$.
When $\F = \{F\}$, we will simply use $F$-free, $\ex(n,F)$, and $\pi(F)$, respectively. 
We say an $n$-vertex $\F$-free $r$-graph $H$ is \emph{extremal} (or an \emph{extremal construction}) for $\F$ if $H$ has $\ex(n, \F)$ hyperedges.  
Determining Tur\'an numbers and the corresponding extremal constructions is one of the central problems in Extremal Combinatorics and has received extensive attention. 

This problem is reasonably well-understood for graphs (the case where $r=2$). The Erd\H os-Stone theorem~\cite{erdos1946structure}, as pointed out by Erd\H os and Simonovits~\cite{erdHos1965limit}, determines the Tur\'an density of all graphs. However, for hypergraphs, the Tur\'an problem is notoriously difficult, and there are only very few exact results. 
The \emph{tetrahedron}, denoted by $K_4^3$, is the complete $3$-graph on four vertices.
It is a famous problem of Tur\'an~\cite{turan1941extremalaufgabe} to determine the Tur\'an density of the tetrahedron. There are exponentially many conjectured extremal constructions by Brown~\cite{K43brown}, Kostochka~\cite{kostochka1982class}, Fon-der-Flaass~\cite{K43Fonderflaass}, and Frohmader~\cite{MR2465761}, see also~\cite{MR2962945} by Razborov.

In this paper, we give the first examples of hypergraphs which share this (conjectured) property with the tetrahedron in the following sense: each of them has exponentially many non-isomorphic extremal constructions and positive Tur\'an density. 
We note that our ext\-remal constructions are close (i.e. $O(n^2)$) to each other in edit-distance, while for the tetrahedron, the (conjectured) extremal constructions differ more from each other (some of them have edit-distance $\Omega(n^3)$).
Following Liu, Mubayi~\cite{LiMu} and Liu, Mubayi, Reiher~\cite{LiMuRe}, Hou, Li, Liu, Mubayi, and Zhang~\cite{hou2022hypergraphs} recently constructed finite families of $3$-graphs such that each family has $\Omega(n)$ non-isomorphic extremal constructions far (i.e. $\Omega(n^3)$) from each other in edit-distance, which is a new phenomenon compared to graphs. 
It remains open to find a single hypergraph with $\omega(1)$ extremal hypergraphs far from each other in edit-distance. We reiterate that the tetrahedron is conjectured to be such a hypergraph.

A classical theorem in hypergraph Tur\'an theory is the following result of Bollob{\'a}s~\cite{bollobas1974three}. The \emph{generalized triangle}, denoted by $F_5$, is the $3$-graph on vertex set $\{1,2,3,4,5\}$ with hyperedges $\{123,124,345\}$, see Figure~\ref{fig::F5}. Let $K_4^-$ be the $3$-graph on vertex set $\{1,2,3,4\}$ with hyperedges $\{123,124,134\}$, and denote by $S(n)$ the complete tripartite $3$-graph on vertex set $[n] \ce \{1,2,\ldots,n\}$ with parts of sizes as equal as possible, see Figure~\ref{fig::Sn}. The number of hyperedges in $S(n)$ is $s(n) \ce \left\lfloor n/3 \right\rfloor \cdot \left\lfloor (n+1)/3 \right\rfloor \cdot \left\lfloor (n+2)/3 \right\rfloor$. Bollob{\'a}s~\cite{bollobas1974three} proved that $\ex(n,\{K_4^{-}, F_5\}) = s(n)$. This result was extended by Frankl and F{\"u}redi~\cite{frankl1983new}, who proved that $\ex(n,F_5) = s(n)$ for $n \ge 3000$, and it was improved again to $n\geq 33$ by Keevash and Mubayi~\cite{keevash2004stability}. 

\begin{figure}[h!]
\tikzstyle{every node}=[circle, draw, fill=black!80, inner sep=0pt, minimum width=2.4pt]
	\begin{minipage}{0.48\textwidth}
	\centering
	\begin{tikzpicture}[scale = 1.4]
		\draw[color=white, opacity = .0, use as bounding box] (0,-1.3) rectangle (2,1.5);
		\node (1) at (0, 0.8)  [label=right:$1$] {};  
		\node (2) at (1, 0.8)  [label=right:$2$] {};
		\node (3) at (2, 0.8)  [label=right:$3$] {};
		\node (4) at (2, 0)  [label=right:$4$] {};
		\node (5) at (2, -0.8) [label=right:$5$] {};
		
		\begin{pgfonlayer}{bg}
            \draw[thick] \hedgeiii{1}{2}{3}{3mm};
		    \draw[thick] \hedgeiii{1}{2}{4}{2.6mm};
		    \draw[thick] \hedgeiii{3}{4}{5}{3mm};
        \end{pgfonlayer}
    \end{tikzpicture}
    \caption{Hypergraph $F_5$.} \label{fig::F5}
    \end{minipage}
	\begin{minipage}{0.48\textwidth}
	\centering
	\begin{tikzpicture}[scale = 1.4]
	    \draw[color=white, opacity = .0, use as bounding box] (0,-1.3) rectangle (2,1.5);
	
	    \node (x) at (1, 0.7) [label=right:] {};  
	    \node (y) at (0.25, -0.6) [label=below:] {};
	    \node (z) at (1.75, -0.6) [label=below:] {};
	    
	    \node at (1.9, 0.8) [fill=black!0,draw=black!0] {$V_1$};
	    \node at (-0.2, 0.2) [fill=black!0,draw=black!0] {$V_2$};
	    \node at (2.2, 0.2) [fill=black!0,draw=black!0] {$V_3$};
	
		\draw [very thick] (1,0.8) ellipse (0.7cm and 0.5cm);
		\draw [very thick] (0,-0.5) ellipse (0.7cm and 0.5cm);
		\draw [very thick] (2,-0.5) ellipse (0.7cm and 0.5cm);
        
        \draw[thick] (x) -- (y) -- (z) -- (x);
		
    \end{tikzpicture}
    \caption{Hypergraph $S(n)$.} \label{fig::Sn}
    \end{minipage}
\end{figure}

For every integer $t \ge 0$, denote by $\ffp$ the hypergraph on vertex set $\{1,2,\ldots, 5+t\}$ with hyperedges $\{123,124\} \cup \{34k : 5 \le k \le 5+t\}$, see Figure~\ref{fig::ffp}. Hence $F_5^0=F_5$. Furthermore, denote by $\ftp$ the hypergraph on vertex set $\{1,2,\ldots, 3+t\}$ with hyperedges $\{12k: 3\le k \le 3+t\}$. 
Note that $\ftp$-free hypergraphs are exactly the hypergraphs with maximum codegree at most $t$, i.e., every pair of vertices is in at most $t$ hyperedges.
Let $\mH(n,t)$ be the family of extremal $n$-vertex $\ftp$-free $3$-graphs $H(n,t)$.
Define $\fHo$ as the family of $3$-graphs $\Ho$ on vertex set $[n]$, where $\Ho$ is obtained from $S(n)$ by adding a copy of $H(|V_i|,t) \in \mH(|V_i|, t)$ to every part $V_i$ in the partition of $S(n)$, see Figure~\ref{fig::Ho}. Let
\bdm
\sHo \ce \pa \pb \pc 
+ \ex\left(\pa,\ftp\right) + \ex\left(\pb,\ftp\right) + \ex\left(\pc,\ftp \right)
\edm
be the number of hyperedges in every $\Ho \in \fHo$. Our main result is the following theorem.

\begin{theorem} \label{thm::Main}
For every integer $t \ge 0$ and sufficiently large $n$, the $\ffp$-free $3$-graphs on vertex set $[n]$ with maximum number of hyperedges are exactly the hypergraphs $\Ho \in \fHo$, and hence $\ex(n, \ffp)=\sHo$.
\end{theorem}

Thus, we solve the Tur\'an problem for each hypergraph in this infinite family of forbidden hypergraphs for sufficiently large $n$ and characterize their extremal constructions. 
Our ext\-remal constructions $\fHo$ are similar in soul to the following result of Simonovits~\cite{simonovits1968method}.
For $r,d \ge 2$, let $Q(r,d)$ be the complete $(d+1)$-partite graph on $rd+1$ vertices with one part of size $1$ and $d$ parts of size $r$. 
Let $\U(n,r,d)$ be the family of graphs, each of which is obtained from the $n$-vertex complete $d$-partite graph with parts of sizes as equal as possible by adding maximum number of edges inside each part such that every vertex is adjacent to at most $r-1$ other vertices in the same part. 
Let $\U^*(n,r,d)$ be the family of graphs in $\U(n,r,d)$ with the extra requirement that there is no triangle inside any part.
Simonovits~\cite{simonovits1968method}, generalizing an unpublished result of Erd\H os, proved that all the graphs in $\U^*(n,r,d)$ are extremal constructions for $Q(r,d)$ and all the extremal constructions for $Q(r,d)$ are in $\U(n,r,d)$, for sufficiently large $n$.

For our proof of Theorem~\ref{thm::Main}, we derive a stability result, followed up by a cleaning method, which reveals the structure of the extremal $\ffp$-free hypergraphs step by step.
\begin{figure}[ht]
    \begin{minipage}[b]{0.48\textwidth}
	\centering
	\begin{tikzpicture}[scale = 1.42 ]
	\tikzstyle{every node}=[circle, draw, fill=black!80, inner sep=0pt, minimum width=2.1pt]
		\draw[color=white, opacity = .0, use as bounding box] (0,-1.3) rectangle (2,1.5);
		\node (1) at (0, 0.8)  [label=right:$1$] {};  
		\node (2) at (1, 0.8)  [label=right:$2$] {};
		\node (3) at (2, 0.8)  [label=right:$3$] {};
		\node (4) at (2, 0)  [label=right:$4$] {};
		\node (5) at (2, -0.8) [label=right:$5$] {};
		\node at (1.4, -0.8) [fill=black!0,draw=black!0] {$\ldots$};
		\node (5+t) at (0.1, -0.8) [label={[label distance=0.05cm]5:$5+t$}] {};
		
		\begin{pgfonlayer}{bg}
            \draw[thick] \hedgeiii{1}{2}{3}{3mm};
		    \draw[thick] \hedgeiii{1}{2}{4}{2.6mm};
		    \draw[thick] \hedgeiii{3}{4}{5}{3mm};
		    \draw[thick] \hedgeiii{3}{4}{5+t}{3mm};
        \end{pgfonlayer}
    \end{tikzpicture}
    \caption{Hypergraph $\ffp$.} \label{fig::ffp}
    \end{minipage}
	\begin{minipage}[b]{0.48\textwidth}
	\centering
	\begin{tikzpicture}[scale = 1.42 ]
	    \draw[color=white, opacity = .0, use as bounding box] (0,-1.3) rectangle (2,1.5);
	
	    \node at (1.9, 0.8) [fill=black!0,draw=black!0] {$V_1$};
	    \node at (-0.2, 0.2) [fill=black!0,draw=black!0] {$V_2$};
	    \node at (2.2, 0.2) [fill=black!0,draw=black!0] {$V_3$};
	    
	    \fill [gray!33] (1,0.8) ellipse (0.7cm and 0.5cm);
	    \fill [gray!33] (0,-0.5) ellipse (0.7cm and 0.5cm);
		\fill [gray!33] (2,-0.5) ellipse (0.7cm and 0.5cm);
		
		\draw [very thick] (1,0.8) ellipse (0.7cm and 0.5cm);
		\draw [very thick] (0,-0.5) ellipse (0.7cm and 0.5cm); 
		\draw [very thick] (2,-0.5) ellipse (0.7cm and 0.5cm); 
		
		\node at (1,0.9) {${\scriptstyle H\left(|V_1|,\,t\right)}$};
		\node at (-0.2,-0.45) {${\scriptstyle H\left(|V_2|,\,t\right)}$};
		\node at (2.2,-0.45) {${\scriptstyle H\left(|V_3|,\,t\right)}$};

		\tikzstyle{every node}=[circle, draw, fill=black!80, inner sep=0pt, minimum width=2.1pt]
	    \node (x) at (1, 0.7) {};  
	    \node (y) at (0.25, -0.6)  {};
	    \node (z) at (1.75, -0.6) {};
        
        \draw[thick] (x) -- (y) -- (z) -- (x);
    \end{tikzpicture}
    \caption{Hypergraph $\Ho$.} \label{fig::Ho}
    \end{minipage}
\end{figure}
One of the roadblocks in our proof is that there exists another family of almost extremal $\ffp$-free $3$-graphs, whose number of hyperedges is smaller than $\sHo$ only by $O(n)$.
Hence for the proof of Theorem~\ref{thm::Main}, we need to have a good understanding of $\ex(n,\ftp)$.
The value of $\ex(n,\ftp)$ is closely related to designs. An $(n,k,r,t)$\emph{-design} is a family $X$ of distinct $k$-subsets of an $n$-set $V$, such that every $r$-subset of $V$ belongs to exactly $t$ elements of $X$. Hence, if every pair of vertices in hypergraph $H(n,t) \in \mH(n,t)$ is in exactly $t$ hyperedges, then $H(n,t)$ is just an $(n,3,2,t)$-design and thus $H(n,t)$ has $\frac{t}{3}\binom{n}{2}$ hyperedges. Dehon~\cite{MR685624} proved the following necessary and sufficient condition for the existence of $(n,3,2,t)$-designs. 
\begin{theorem}[Dehon~\cite{MR685624}]
\label{designdehon}
Let $t$ and $n$ be integers such that $0\leq  t \leq n-2$. Then, there exists an $(n,3,2,t)$-design iff $t n (n-1)\equiv 0 \pmod 6$ and $t (n-1) \equiv 0 \pmod 2$.
\end{theorem}
If $t$ is a multiple of $6$, then there always exists an $(n,3,2,t)$-design for sufficiently large $n$. Keevash proved that the number of non-isomorphic $(n,3,2,t)$-designs (see Theorem 6.1 in~\cite{keevash2018counting}) is $$
t!^{-\b{n}{2}} \left( \left(\frac{t}{e}\right)^2(n-2) + o(n-2)\right)^{\frac{t}{3} \b{n}{2}}=n^{\Omega(n^2)}.
$$ Thus, by Theorem~\ref{thm::Main}, there are exponentially many extremal constructions for $F_5^{6t}$ whenever $t \ge 1$.

Note that if $n = 6k+1$ or $n= 6k +3 $ for some integer $k \ge (t + 1)/6$, then there always exists an $(n,3,2,t)$-design. 
Thus, we have the following corollary of Theorem~\ref{thm::Main}.
\begin{corollary} 
For every $t \ge 0$, if $n$ is sufficiently large and $n\equiv 3$ or $9 \pmod {18}$, then 
$$
\ex(n,\ffp) = \frac{1}{27}n^3 + \frac{1}{18}tn^2 - \frac{1}{6}tn.
$$
\end{corollary}

There does not exist an $(n,3,2,t)$-design for every $n$ and $t$, and therefore it is non-trivial to give a reasonably good lower bounds on $\ex(n,\ftp)$ in general. In Section~\ref{subsec::ressHo}, we derive a lower bound on $\ex(n,\ftp)$ which is good enough for our purpose.
We cannot give an explicit formula for $\ex(n,\ftp)$. Similar phenomena often occur in Extremal Combinatorics. An interesting classical example is by Erd\H os and Simonovits~\cite{erdHos1971extremal}. Denote by $K_{2,2,2}$ the complete tripartite graph with each part of size $2$. They~\cite{erdHos1971extremal} proved that for sufficiently large $n$, the extremal $K_{2,2,2}$-free graph is realized by the graphs whose vertex set can be partitioned into two subsets $A$ and $B$ such that $G[A]$ is $C_4$-free and $G[B]$ is $P_3$-free, where $C_4$ is the cycle with $4$ vertices and $P_3$ is the path with $3$ vertices. However, they did not determine the exact sizes of $A$ and $B$, as $\ex(n,C_4)$ is not known precisely for every $n$. Similarly to $\ex(n,C_4)$, in our case $\ex(n, \ftp)$ is not known precisely for every $n$. We are still able to confirm the above-mentioned unified description of the extremal $\ffp$-free hypergraphs and in particular, prove that the three parts in the partition have to be balanced, i.e., they have sizes $\left \lfloor n/3 \right \rfloor$, $\left\lfloor (n+1)/3 \right\rfloor$ and $\left\lfloor (n+2)/3 \right\rfloor$.

Our paper is organized as follows. In Section~\ref{sec::Pre}, we introduce our notation, give an estimate on $\ex(n,\ftp)$ and some useful results about $\Ho$, and deduce a stability result for $\ffp$. In Section~\ref{sec::Pro}, we prove Theorem~\ref{thm::Main}, our main result. We discuss some open problems in Section~\ref{sec::Con}.

\section{Preliminaries} \label{sec::Pre}

\subsection{Notation} \label{subsec::Not}
For a positive integer $m$, we use $[m]$ for the set $\{1,2,\ldots, m\}$. We write $xy$ for the set $\{x,y\}$ and $xyz$ for the set $\{x,y,z\}$. For a set $S$ and a non-negative integer $k$, we use $\b{S}{k}$ to denote the family of all $k$-subsets of $S$. For an $n$-vertex $r$-graph $H$, we will assume its vertex set $V(H)$ is $[n]$. We often use $H$ for the hyperedge set of $H$ and denote by $|H|$ the number of hyperedges in $H$. Let $H$ be an $n$-vertex $3$-graph. For vertices $x,y\in [n]$ and (not necessarily disjoint) subsets of vertices $S,T\subseteq [n]$, let
\begin{itemize}
    \item $L_{S,T}(x) \ce \{ yz : y\in S,\,z\in T,\,xyz\in H\}$ be the \emph{link graph} of $x$ between $S$ and $T$,
    \item $d_{S,T}(x) \ce |L_{S,T}(x)|$ be the \emph{degree} of $x$ between $S$ and $T$,
    \item $L_{S}(x,y) \ce \{z : z\in S,\, xyz \in H\}$ be the set of \emph{neighbors} of $x$ and $y$ in $S$, and
    \item $d_S(x,y) \ce |L_{S}(x,y)|$ be the \emph{codegree} of $x$ and $y$ in $S$. 
\end{itemize}   
For a given partition $\pi = (V_1,V_2,V_3)$ of a subset of $[n]$, we say that $\pi$ is \emph{balanced} if the sizes of its parts differ by at most $1$. We denote by $K_\pi$ the set of triples with exactly one vertex in each part of $\pi$. Let $H_\pi \ce H \cap K_\pi$ and $\bar{H}_\pi \ce K_\pi \sm H_\pi$. We call hyperedges in $H_\pi$ the \emph{crossing hyperedges} (of $\pi$) and call hyperedges in $\bar{H}_\pi$ the \emph{missing crossing hyperedges} (of $\pi$). For $i,j \in \{1,2,3\}$, we simply write $L_{i,j}(x)$ for $L_{V_i,V_j}(x)$ and $L_{i}(x,y)$ for $L_{V_i}(x,y)$. Similarly, we use $d_{i,j}(x)$ for $d_{V_i,V_j}(x)$ and $d_i(x,y)$ for $d_{V_i}(x,y)$. We also write $L(x,y)$ for $L_{[n]}(x,y)$, $L(x)$ for $L_{[n],[n]}(x)$, $d(x,y)$ for $d_{[n]}(x,y)$, and $d(x)$ for $d_{[n],[n]}(x)$.

\subsection{Results on $\boldsymbol{\sHo}$} \label{subsec::ressHo}
\begin{proposition}
\label{pro::SteineralmostTuran3}
For every integer $t \ge 0$ and sufficiently large $n$, we have \begin{align*}
  \frac{t}{3}\binom{n}{2}  \geq
  \ex(n,\ftp) \geq
\begin{cases}
\frac{t}{3}\binom{n}{2}-\frac{n}{3}-\frac{t^2}{6} - 3t& \text{if } t \text{ is odd and } n \text{ is even},  \\
\frac{t}{3}\binom{n}{2}-\frac{8t}{3} & \text{if } t \text{ is odd and } n \text{ is odd},  \\
\frac{t}{3}\binom{n}{2}-\frac{2t}{3} &\text{if } t \text{ is even}.
\end{cases}
\end{align*}
\end{proposition}
\noindent
We give the proof of Proposition~\ref{pro::SteineralmostTuran3} in the Appendix.

\label{subsec::com}
\begin{lemma} \label{lem::LowerHo}
For every integer $t \ge 0$ and sufficiently large $n$, we have 
\bdm
\sHo \in \left[
\frac{1}{27}n^3 + \frac{1}{18}tn^2 -\left(\frac{1}{6}t+\frac{4}{9}\right)n - \left(\frac{1}{2}t^2 + \frac{80}{9}t + \frac{2}{27}\right),\, 
\frac{1}{27}n^3 + \frac{1}{18}tn^2 - \frac{1}{6}tn 
\right].
\edm
\end{lemma}
\begin{proof}
By the definition of $\sHo$ and Proposition~\ref{pro::SteineralmostTuran3}, we have the following statements.\\
\bf{Case:}  $n\equiv 0 \pmod 3$
\begin{align*}
\sHo &= \left(\frac{n}{3}\right)^3 + 3\cdot \ex\left(\frac{n}{3}, \ftp\right) \\ &\in 
\left[ \frac{1}{27}n^3 + \frac{1}{18}tn^2 -\left(\frac{1}{6}t+\frac{1}{3}\right)n -\left(\frac{1}{2}t^2 + 9t\right),\,
\frac{1}{27}n^3 + \frac{1}{18}tn^2 - \frac{1}{6}tn\right].
\end{align*}
\bf{Case:} $n\equiv 1 \pmod 3$
\begin{align*}
\sHo 
&= \left(\frac{n-1}{3}\right)^2\left(\frac{n+2}{3}\right) + 2\cdot \ex\left(\frac{n-1}{3}, \ftp \right) + \ex\left(\frac{n+2}{3}, \ftp\right) \\
&\in 
\Bigg[ \frac{1}{27}n^3 + \frac{1}{18}tn^2 - \left(\frac{1}{6}t+\frac{4}{9}\right)n - \left(\frac{1}{2}t^2 + \frac{80}{9}t - \frac{2}{27}\right),\, \\
&\qquad \qquad \qquad \frac{1}{27}n^3 + \frac{1}{18}tn^2 - \left(\frac{1}{6}t + \frac{1}{9}\right)n + \left( \frac{1}{9}t + \frac{2}{27} \right)
\Bigg].    
\end{align*}
\bf{Case:} $n\equiv 2 \pmod 3$ 
\begin{align*}
\sHo 
&= \left(\frac{n-2}{3}\right)\left(\frac{n+1}{3}\right)^2 + \ex\left(\frac{n-2}{3},\ftp\right) + 2\cdot \ex\left(\frac{n+1}{3}, \ftp \right) \\
&\in 
\Bigg[ \frac{1}{27}n^3 + \frac{1}{18}tn^2 - \left(\frac{1}{6}t+\frac{4}{9}\right)n - \left(\frac{1}{2}t^2 + \frac{80}{9}t + \frac{2}{27}\right),\, \\
&\qquad \qquad \qquad \frac{1}{27}n^3 + \frac{1}{18}tn^2 - \left(\frac{1}{6}t + \frac{1}{9}\right)n + \left( \frac{1}{9}t - \frac{2}{27} \right)
\Bigg].    
\qedhere
\end{align*}
\end{proof}

For a partition $\pi = (V_1,V_2,V_3)$ of $[n]$, let $\fHop$ be the family of hypergraphs on vertex set $[n]$ whose hyperedges contain $K_\pi$, and each part $V_i$ spans a copy of $H(|V_i|,t) \in \mH(|V_i|, t)$. Hence, $\fHop$ is just $\fHo$ when $\pi$ is balanced. Let 
\begin{equation*}
\sHop \ce |V_1||V_2||V_3| + \sum_{i=1}^3 \ex(|V_i|, \ftp)
\end{equation*}
be the number of hyperedges in every $\Hop \in \fHop$.

\begin{lemma} \label{lem::balLar}
Let $t \ge 0$. For every sufficiently large integer $n$, we have $\sHo \ge \sHop$ for every partition $\pi$ of $[n]$, where equality holds only if $\pi$ is balanced.
\end{lemma}
\begin{proof}
Suppose that there is a partition $\pi = (V_1,V_2,V_3)$ of $[n]$ that achieves the maximum value of $\sHop$, but $\pi$ is not balanced. Assume that the three parts of $\pi$ have sizes $a_1\le a_2\le a_2$, where $a_1+a_2+a_3 =n$ and $a_1 +2 \le a_3$.  We first prove $a_2 \ge (\frac{1}{3}-0.01)n$.
Otherwise, by Proposition~\ref{pro::SteineralmostTuran3} and Lemma~\ref{lem::LowerHo}, we have
\begin{align*}
\sHop 
&\le a_1a_2a_3 + \frac{t}{3}\left(\b{a_1}{2}+\b{a_2}{2}+\b{a_3}{2}\right)
\le \frac{a_1+a_3}{2}a_2\frac{a_1+a_3}{2}+t\b{n}{2} \\
&= a_2\left(\frac{n-a_2}{2}\right)^2 + t\b{n}{2}
\le \left(\frac{1}{3}-0.01\right)n\left(\frac{n-\left(\frac{1}{3}-0.01\right)n}{2}\right)^2 + t\b{n}{2} \\
&< \left(\frac{1}{27}- 10^{-10}\right)n^3 < \sHo,
\end{align*}
where the third-to-last inequality holds because the function $x \to x((n-x)/2)^2$ is monotone increasing for $x \le n/3$. This contradicts our choice of $\pi$. Now, let $\pi'$ be the partition of $[n]$ obtained from $\pi$ by moving a vertex from $V_3$ to $V_1$. Then,
\begin{displaymath}
s_{t,\pi'}^\supforOne(n) - s_{t,\pi}^\supforOne(n) = 
(a_3-a_1-1)a_2 + \ex(a_1+1,\ftp) + \ex(a_3-1,\ftp) - \ex(a_1,\ftp)  - \ex(a_3,\ftp).
\end{displaymath}
By Proposition~\ref{pro::SteineralmostTuran3},
\begin{align*}
    \ex(a_1+1,\ftp) - \ex(a_1,\ftp)\geq  \frac{t}{3} \b{a_1+1}{2} - \frac{a_1+1}{3} - \frac{t^2}{6} - 3t - \frac{t}{3} \b{a_1}{2} \geq \frac{t}{3}a_1- \frac{a_1+1}{3}-4t^2
\end{align*}
and 
\begin{align*}
\ex(a_3-1,\ftp)  - \ex(a_3,\ftp)&\geq \frac{t}{3} \b{a_3-1}{2} - \frac{a_3-1}{3} -  4t^2 - \frac{t}{3} \b{a_3}{2}= -\frac{t+1}{3}(a_3-1)- 4t^2.
\end{align*}
We conclude 
\begin{align*}
    s_{t,\pi'}^\supforOne(n) - s_{t,\pi}^\supforOne(n) &\geq (a_3-a_1-1)\left(a_2 -\frac{t}{3}\right) -\frac{a_1+a_3}{3} - 8t^2
    \geq \left(a_2 -\frac{t}{3}\right) -\frac{a_1+a_3}{3} - 8t^2\\
    &=\frac{4}{3}a_2 - \frac{n}{3}-\frac{t}{3}-8t^2
>0,
\end{align*}
contradicting our choice of $\pi$.
\end{proof}

\subsection{An almost extremal $\boldsymbol{\ffp}$-free hypergraph} \label{subsec::AnAlo}
For $T = \{v_1,\ldots,v_t\} \subseteq [n]$ and a partition $\varpi = (W_1,W_2,W_3)$ of $[n]\sm T$, define $\Htp$ to be the $3$-graph on vertex set $[n]$ with hyperedges 
\bdm
K_\varpi\cup \b{T}{3} \cup \bigcup_{i=1}^3\{v x_1x_2: v \in T,\, x_1,x_2 \in W_i\}.
\edm
For an illustration of $\Htp$, see Figure~\ref{fig::Ht}. Let 
\bdm
\sHtp \ce |W_1||W_2||W_3| + t\sum_{i=1}^3 \b{|W_i|}{2} + \b{t}{3}
\edm
be the number of hyperedges in $\Htp$. We remark that similar constructions were proved by Simonovits to be extremal constructions for certain graphs (see Theorem 2 in~\cite{simonovits1968method}).
\begin{figure}[ht]
\tikzstyle{every node}=[circle, draw, fill=black!80, inner sep=0pt, minimum width=4pt]
    \centering
	\begin{tikzpicture}[scale = 1.45]
	    \draw[color=white, opacity = .0, use as bounding box] (0,-1) rectangle (2,1.5);
	    
	    \tikzstyle{every node}=[circle, draw, fill=black!80, inner sep=0pt, minimum width=2.1pt]
	    \node (x) at (1, 0.7) [label=right:] {};  
	    \node (y) at (0.25, -0.6) [label=below:] {};
	    \node (z) at (1.75, -0.6) [label=below:] {};
	    
	    \node (v1) at (2.45, 1.1) [label={[label distance=0.05cm]0:$v_1$}] {};  
	    \node (v2) at (2.45, 0.9) [label={[label distance=0.05cm]0:$v_2$}] {};
	    \node (v3) at (2.45, 0.7) [label={[label distance=0.05cm]0:$v_3$}] {};
	    \node (vt) at (2.45, 0.3) [label={[label distance=0.05cm]0:$v_t$}] {};
	    
	    \node at (1.75, 1.2) [fill=black!0,draw=black!0] {$V_1$};
	    \node at (-0.2, 0.2) [fill=black!0,draw=black!0] {$V_2$};
	    \node at (2.95, -0.5) [fill=black!0,draw=black!0] {$V_3$};
	    
	    \node (x1) at (1.32, 0.6) {};
	    \node (x2) at (1.35, 0.45) {};
	    \node (y1) at (0.45, -0.4) {};
	    \node (y2) at (0.5,-0.5) {};
	    \node (z1) at (1.9, -0.4) {};
	    \node (z2) at (2, -0.46) {};

		\draw [very thick] (1,0.8) ellipse (0.7cm and 0.5cm);
		\draw [very thick] (0,-0.5) ellipse (0.7cm and 0.5cm);
		\draw [very thick] (2,-0.5) ellipse (0.7cm and 0.5cm);

		\begin{pgfonlayer}{bg}
            \node at (2.45, 0.55) [fill=black!0,draw=black!0] {$\vdots$};
            \draw[thick] \hedgeiii{v1}{v2}{v3}{0.8mm};
            
            \draw[thick] (x) -- (y) -- (z) -- (x);
            \draw[thick] (v3) -- (x1) -- (x2) -- (v3);
            \draw[thick] (v3) -- (y1) -- (y2) -- (v3);
            \draw[thick] (v3) -- (z1) -- (z2) -- (v3);
        \end{pgfonlayer}
    \end{tikzpicture}
    \caption{Hypergraph $\Htp$.} \label{fig::Ht}
\end{figure}

\begin{lemma} \label{lem::HoHt}
Let $t \ge 1$. For every sufficiently large integer $n$, we have $\sHo > \sHtp + \frac{n}{10}$ for every set $\{v_1, \ldots, v_t\} \subseteq [n]$ and partition $\varpi$ of $[n] \sm\{v_1,\ldots, v_t\}$.
\end{lemma}
\begin{proof}
Define function
\bdm
f_{n,t}(x_1,x_2,x_3) \ce x_1x_2x_3 + t\left(\frac{x_1(x_1-1)}{2}+\frac{x_2(x_2-1)}{2}+\frac{x_3(x_3-1)}{2}\right) + \b{t}{3}
\edm 
on the domain $D:=\{(x_1,x_2,x_3)\in\mathbb{R}^3: \ x_1,x_2,x_3 \ge 0, \ x_1+x_2+x_3 = n-t\}$. Note that for every set $T = \{v_1,\ldots,v_t\}\subseteq [n]$ and partition $\varpi = (W_1,W_2,W_3)$ of $[n] \sm T$, we have $\sHtp = f(|W_1|,|W_2|,|W_3|)$. Also note that
$$
f_{n,t}\left(\frac{n-t}{3},\frac{n-t}{3},\frac{n-t}{3}\right) = \frac{1}{27}n^3 + \frac{1}{18}tn^2 - \left(\frac{2}{9}t^2 + \frac{1}{2}t\right) n + \frac{8}{27}t^3 + \frac{1}{3}t.
$$
Since $D$ is compact, $f_{n,t}$ achieves its maximum value at some point $(y_1,y_2,y_3)\in D$. If $y_1 \le t$, then  
$$
f_{n,t}(y_1,y_2,y_3) \le tn^2 + 3t\frac{n^2}{2} +t^3 < \frac{n^3}{27} < f_{n,t}\left(\frac{n-t}{3},\frac{n-t}{3},\frac{n-t}{3}\right),
$$
a contradiction. Thus, $y_1 > t$. Similarly, $y_2 > t$ and $y_3 > t$. We have
\begin{align*}
f_{n,t}\left(y_1,\frac{y_2+y_3}{2},\frac{y_2+y_3}{2}\right) - f_{n,t}\left(y_1,y_2,y_3\right)
= \frac{1}{4}(y_1-t)(y_2-y_3)^2 \ge 0,
\end{align*}
with equality iff $y_2 = y_3$. Using the symmetry of $f$, we conclude that the maximum value of $f_{n,t}$ is achieved when $x_1=x_2=x_3=\frac{n-t}{3}$. For $t \ge 1$, by Lemma~\ref{lem::LowerHo}, we have 
$$
\sHo - f_{n,t}\left(\frac{n-t}{3},\frac{n-t}{3},\frac{n-t}{3}\right) 
\ge \left(\frac{2}{9}t^2 + \frac{1}{3}t - \frac{4}{9} \right)n - \left(\frac{8}{27}t^3 + \frac{1}{2}t^2 + \frac{83}{9}t + \frac{2}{27}\right)
> \frac{n}{10}.
$$
Therefore, $\sHo > \sHtp + \frac{1}{10}n$.
\end{proof}

\subsection{Stability result} \label{subsec::sta}
In this section, we prove the following stability theorem for $\ffp$.
\begin{theorem}[Stability theorem for $\ffp$] \label{thm::o3Sta}
For every $t \ge 0$ and $\eps >0$, there exist $\delta \in (0, \eps)$ and $N>0$ such that for every $n>N$, the following statement holds. Let $H$ be an $\ffp$-free $3$-graph $H$ on vertex set $[n]$ with at least $(\frac{1}{27}-\delta)n^3$ hyperedges, and let $\pi= (V_1,V_2,V_3)$ be a partition of $[n]$ that maximizes $|H_\pi|$. Then $|H \sm H_\pi| \le \eps n^3$. 
\end{theorem}

\begin{claim} \label{cla::o3staP}
The partition $\pi=(V_1,V_2,V_3)$ in Theorem~\ref{thm::o3Sta} has the following additional pro\-perties.
\begin{enumerate}[label=(\roman*)]
    \item For $i \in [3]$, we have $(\frac{1}{3} - 3\eps^{1/2})n< |V_i| < (\frac{1}{3} + 6\eps^{1/2})n $.

    \item For $\{i,j,k\} = [3]$, we have
        $
        |\{ xy : x \in V_i,\,y\in V_j,\,d_k(x,y)\le |V_k|-\eps^{1/2}n \}| \le 2\eps^{1/2} n^2.
        $
    \item For $\{i,j,k\} = [3]$ and every $x\in V_i$, we have
            $
            d_{i,j}(x),\;d_{i,k}(x) \le 3\eps^{1/2} n^2.
            $
\end{enumerate}
\end{claim}

We need the following theorems and definitions for proving Theorem~\ref{thm::o3Sta}. For an $r$-graph $H$, a \emph{blow-up} of $H$ is an $r$-graph obtained from $H$ by replacing every vertex $v$ by a set $S_v$, where $S_{v_1} \cap S_{v_2} = \es$ for distinct vertices $v_1,v_2$ and replacing every hyperedge $\{v_1,\ldots, v_r\}$ by the complete $r$-partite $r$-graph on $S_{v_1}\cup\ldots\cup S_{v_r}$. When $|S_v|=t$ for every vertex $v$, the blow-up is called the $t$-\emph{blow-up} of $H$ and denoted by $H[t]$. Note that $\ffp$ is a blow-up of $F_5$ and $\ffp \subseteq F_5[t+1]$ for every $t \ge 0$. We will see that Theorem~\ref{thm::o3Sta} follows easily from the stability result for $F_5$ and the fact that blow-ups inherit the stability of the original hypergraph. 

\begin{theorem}[Stability result for $F_5$, Keevash and Mubayi~\cite{keevash2004stability}]
\label{thm::F5Sta}
For every $a >0$, there exist $b>0$ and $N>0$ such that if $n>N$ and $H$ is an $F_5$-free $3$-graph on vertex set $[n]$ with at least $(\frac{1}{27}-b)n^3$ hyperedges, then there exists a partition $\pi = (V_1,V_2,V_3)$ of $[n]$ such that $|H \sm H_\pi| \le a n^3$.
\end{theorem}

\begin{lemma}[e.g.~\cite{keevash2011hypergraph}] \label{lem::ManCopOfFThenBlo}
For fixed integers $l \ge r \ge 2$, $t \ge 1$ and every $\alpha > 0$, there exists $n_0 = n_0(l,r,t,\alpha)$ such that the following holds. Let $F$ be an $r$-graph on $l$ vertices and $H$ be an $r$-graph on $n\ge n_0$ vertices. If $H$ contains at least $\alpha n^l$ copies of $F$, then $H$ contains a copy of $F[t]$.
\end{lemma}

We also use the hypergraph removal lemma.
\begin{theorem}[R{\"o}dl, Nagle, Skokan, Schacht, and Kohayakawa~\cite{rodl2005hypergraph}, Gower~\cite{gowers2007hypergraph}] \label{thm::HypRemLem}
For fixed integers $l \ge r \ge 2$ and every $\mu > 0$, there exists $\zeta = \zeta(l,r,\mu) >0$ and $n_0 = n_0(l,r,\mu)$ such that the following holds. Let $F$ be an $r$-graph on $l$ vertices and $H$ be an $r$-graph on $n \ge n_0$ vertices. If $H$ contains at most $\zeta n^l$ copies of $F$, then one can delete $\mu n^r$ hyperedges of $H$ to make it $F$-free.
\end{theorem}

\begin{proof}[Proof of Theorem~\ref{thm::o3Sta}]
Assume that $n$ is sufficiently large. 
Let $b$ be given by Theorem~\ref{thm::F5Sta} when applying it for $a = \eps/2$. Let $\delta := \min(b/2, \eps/2)$. Let $H$ be an $\ffp$-free $3$-graph on vertex set $[n]$ with at least $(\frac{1}{27} - \delta)n^3$ hyperedges and let $\pi$ be a $3$-partition of $[n]$ maximizing $|H_\pi|$.

Setting $\mu = \delta$, $l = 5$, and $r=3$ in Theorem~\ref{thm::HypRemLem}, we get that there exists $\zeta>0$ such that if an $n$-vertex $3$-graph $F$ contains at most $\zeta n^5$ copies of $F_5$, then we can delete at most $\delta n^3$ hyperedges to make it $F_5$-free. Setting $\alpha = \zeta$ in Theorem~\ref{lem::ManCopOfFThenBlo}, we get that if $H$ contains at least $\zeta n^5$ copies of $F_5$, then $H$ contains a copy of $F[t+1] \supseteq \ffp$, a contradiction to that $H$ is $\ffp$-free, so $H$ indeed contains at most $\zeta n^5$ copies of $F_5$. Therefore, $H$ contains a spanning $F_5$-free subhypergraph $H'$ with at least $|H|-\delta n^3 \ge (\frac{1}{27} - 2\delta)n^3\ge (\frac{1}{27}-b) n^3$ hyperedges. By Theorem~\ref{thm::F5Sta}, there exists a $3$-partition $\pi'$ of $[n]$ such that $|H' \sm H'_{\pi'}| \le \eps n^3/2$. Thus,
\bdm
|H \sm H_\pi|\le |H \sm H_{\pi'}| \le |H \sm H'| + |H' \sm H_{\pi'}| = |H \sm H'| + |H' \sm H'_{\pi'}| \le \eps n^3. \qedhere
\edm

\end{proof}

\begin{proof}[Proof of Claim~\ref{cla::o3staP}]
For (i), assume $|V_i| \le (\frac{1}{3} - 3\eps^{1/2})n$ for some $i\in[3]$. The function $x \to x((n-x)/2)^2$ is monotone increasing when $x < n/3$, and therefore 
$$
|H|\leq |V_1||V_2||V_3| +|H \sm H_\pi| \le |V_i| \left( \frac{n-|V_i|}{2}\right)^2 +\varepsilon n^3
\le \left(\frac{1}{27}- \frac{5}{4}\eps - \frac{27}{4}\eps^{3/2}\right)n^3 < |H|,
$$
a contradiction.
Thus $|V_i| > (\frac{1}{3} - 3\eps^{1/2})n $ for every $i \in [3]$. Since $|V_1|+ |V_2|+|V_3| = n$, we have $|V_i|< (\frac{1}{3} + 6\eps^{1/2})n$ for $i \in [3]$.

Assume that (ii) does not hold. Then $|\bar{H}_\pi| > \eps^{1/2}n \cdot 2\eps^{1/2}n^2 = 2\eps n^3$, and thus
\begin{align*}
|H| = |V_1||V_2||V_3| - |\bar{H}_\pi| + |H \sm H_\pi|
\le \left(\frac{n}{3}\right)^3 - 2\eps n^3 + \eps n^3
< |H|,
\end{align*}
a contradiction.

For (iii), assume for contradiction that $d_{i,j}(x) > 3\eps^{1/2} n^2$ for some $x\in V_i$. Then, by (ii), there exist $x'\in V_i$ and $y\in V_j$ such that $xx'y\in H$ and $d_k(x',y) > |V_k| - \eps^{1/2} n $. By our assumption that $\pi$ maximizes $|H_\pi|$, we have $d_{j,k}(x) \ge d_{i,j}(x) > 3\eps^{1/2} n^2$; otherwise moving $x$ to $V_k$ will strictly increase $|H_\pi|$. There exists a vertex $z \in V_k$ such that $x'yz \in H$ and $d_j(x,z) \ge t+2$; otherwise for every $z$ with $x'yz \in H$ we have $d_j(x,z) \le t+1$ and then
\begin{align*}
d_{j,k}(x) 
\le
(t+1) d_k(x',y) + |V_j| \left( |V_k| - d_k(x',y)  \right)  \le
(t+1)n +n \cdot \eps^{1/2} n < 3\eps^{1/2} n^2,
\end{align*}
a contradiction. Now choose distinct vertices $y_1,\ldots, y_{t+1} \in L_j(x,z) \sm \{y\}$. The hyperedges $\{x'yx, x'yz, xzy_1, \ldots, xzy_{t+1}\}$ form a copy of $\ffp$ in $H$, a contradiction.
\end{proof}

\section{Proof of Theorem~\ref{thm::Main}} \label{sec::Pro}
In Section~\ref{subsec::O2}, using the stability result for $\ffp$, Theorem~\ref{thm::o3Sta}, we show that if $H$ is an extremal $\ffp$-free hypergraph, then $|\bar{H}_\pi| = O(n^2)$ for some partition $\pi$. 
For a hypergraph $H$ on vertex set $[n]$, denote by $ \delta(H) \ce \min\{d(v):v \in [n]\}$ the \emph{minimum degree} of $H$.
In Sections~\ref{subsec::typ},~\ref{subsec::cle}, and~\ref{subsec::proMin}, we prove Theorem~\ref{thm::Main} with an extra assumption on the minimum degree, by distinguishing ``typical'' and ``non-typical'' vertices and using their respective properties. Finally, in Section~\ref{subsec::proMai}, we derive Theorem~\ref{thm::Main}.

\subsection{Upper bound on $\boldsymbol{|\bar{H}_\pi|}$} \label{subsec::O2}

\begin{lemma} \label{lem::Qn29edges}
Let $H$ be an $\ffp$-free $3$-graph on vertex set $[n]$. For every pair of vertices $x,x'\in [n]$ with $d(x,x') \ge t+1$ and disjoint subsets of vertices $S,T \subseteq [n]$, we have 
\bdm
d_{S,T}(x) + d_{S,T}(x') \le |S||T| + (t+3)n
\edm
and 
\bdm
d_{S,S}(x) + d_{S,S}(x')  \le \frac{1}{2} |S|^2 + (t+4)n.
\edm
\end{lemma}
\begin{proof}
Fix $t+1$ distinct vertices $v_1,\ldots, v_{t+1} \in L(x,x')$. There are at least $(|S|-t-3)(|T|-t-3)$ pairs $yz$ such that $y \in S$, $z \in T$, and $\{y,z\} \cap \{x,x',v_1,\ldots,v_{t+1}\} = \es$. For each such pair, $yzx\notin H$ or $yzx'\notin H$; otherwise the hyperedges $\{yzx$, $yzx'$, $xx'v_1$, \ldots, $xx'v_{t+1}\}$ form a copy of $\ffp$ in $H$. Then, 
$$
d_{S,T}(x) + d_{S,T}(x') \le 2|S||T| - (|S|-t-3)(|T|-t-3)  \le |S||T| + (t+3)n. 
$$
Similarly, there are at least $\b{|S| - t- 3}{2}$ pairs $yz$ such $y, z\in S$ and $\{y,z\} \cap \{x,x',v_1,\ldots, v_{t+1}\} = \es$. Again, for each such pair, $yzx\notin H$ or $yzx'\notin H$. We conclude
\begin{align*}
d_{S,S}(x) + d_{S,S}(x') 
&\le 2\b{|S|}{2} - \b{|S| - t- 3}{2}  \le \frac{1}{2}|S|^2 + (t+4)n. \qedhere
\end{align*}

\end{proof}

\begin{theorem}\label{thm::UpperHpi}
For every $t \ge 0$ and $\eps >0$, there exists $N>0$ such that if $n>N$ the following holds. Let $H$ be an $\ffp$-free $3$-graph on vertex set $[n]$ with $|H| \ge \sHo$ and $\pi$ be a partition of $[n]$ maximizing $|H_\pi|$. Then,
$
|\bar{H}_\pi| \le  \left( \frac{1}{9} + \eps \right)tn^2.
$
\end{theorem}

\begin{proof}
We can assume that $\eps$ is sufficiently small and $n$ is sufficiently large. Let $\beta>0$ be the real number such that $\eps = (2t+25)\beta^{1/2}$.
For $i\in [3]$, we define 
\begin{gather*}
P_i \ce \{ xx' : x,x'\in V_i,\,d(x,x') \le t\}
\quad
\textrm{and} 
\quad
Q_i \ce \{ xx' : x,x'\in V_i,\,d(x,x') \ge t+1\}.
\end{gather*}
Let $M_i$ be a maximum matching of $Q_i$, and let $E_i$ be the set of endpoints of the pairs in $M_i$, i.e., the vertices $x \in V_i$ such that there exists $x' \in V_i$ with $xx' \in M_i$. Note that $|E_i| = 2 |M_i|$ and every pair in $Q_i$ contains at least one point in $E_i$, since $M_i$ is a maximum matching of $Q_i$.

\begin{claim} \label{cla::QAsmall}
We have $|Q_i| \le (t+2) \beta^{1/2} n^2$ and $|M_i| \le (t+2) \beta^{1/2} n$ for $i\in [3]$.
\end{claim}
\begin{proof}
For $\{i,j,k\} = [3]$, define $S_{j,k}$ to be the set of pairs $yz$ such that $y \in V_j, z \in V_k$, and $d_i(y,z) \ge |V_i| - \beta^{1/2}n$.
By Claim~\ref{cla::o3staP} (i) and (ii),  
$$
|S_{j,k}| \ge |V_j||V_k| - 2\beta^{1/2}n^2 \geq\left(\left(\frac{1}{3} - 3\beta^{1/2}\right)n\right)^2 - 2\beta^{1/2}n^2 \ge \frac{n^2}{10}.
$$
Thus, we can choose $t+2$ pairwise disjoint pairs $y_1z_1,\ldots, y_{t+2}z_{t+2}$ from $S_{j,k}$. Define $U_i$ to be the set of vertices $x\in V_i$ such that $xy_az_a \notin H$ for some $a \in [t+2]$. Then, $|U_i|\leq (t+2)\beta^{1/2}n$.

For $x_1,x_2 \in V_i\setminus U_i$, we have $x_1x_2\notin Q_i$; otherwise there are $t+1$ distinct vertices $v_1, \ldots, v_{t+1} \in L(x_1,x_2)$ and there is some $a\in [t+2]$ such that $\{v_1,\ldots, v_{t+1}\} \cap \{y_a,z_a\} = \es$, and hence, the hyperedges $\{y_az_ax_1, y_az_ax_2, x_1x_2v_1, \ldots, x_1x_2v_{t+1}\}$ form a copy of $\ffp$ in $H$, a contradic\-tion. 
Therefore, every pair in $Q_i$ contains at least one vertex in $U_i$.  
We conclude $|Q_i| \le |U_i||V_i|\leq (t+2) \beta^{1/2} n^2$ and
$|M_i| \le |U_i| \le (t+2) \beta^{1/2} n$.
\end{proof}

\begin{claim} \label{cla::uppHsmHpi}
We have
$$
|H \sm H_\pi| \le \left(\frac{1}{6} + 6\beta^{1/2} + 54\beta\right) tn^2 + |M|(2t+16) \beta^{1/2} n^2 .
$$
\end{claim}
\begin{proof}
For $i\in [3]$, let $H_i$ be the set of hyperedges in $H$ containing at least two vertices in $V_i$. The number of hyperedges in $H_i$ that contain a pair in $P_i$ is most $t|P_i|$, because every pair in $P_i$ has codegree at most $t$ by definition. The number of hyperedges in $H_i$ that contain no pair in $P_i$ and also no vertex in $[n] \sm V_i$ is at most $|E_i||Q_i|$. The number of hyperedges in $H_i$ that contain no pair in $P_i$ but a vertex in $[n]\sm V_i$ is at most $|E_i|\cdot 2 \cdot 3\beta^{1/2}n^2$ by Claim~\ref{cla::o3staP} (iii). Adding up the number of hyperedges of these three types, we get
\begin{align*}
|H_i| 
&\le t\b{\left(\frac{1}{3}+ 6\beta^{1/2}\right)n}{2} + 2(t+2)\beta^{1/2}n^2 |M_i| + 12\beta^{1/2}n^2 |M_i| \\
&\le \left(\frac{1}{18} + 2\beta^{1/2} + 18\beta\right) tn^2 + (2t+16) \beta^{1/2} n^2 |M_i|,    
\end{align*}
where the first inequality follows from Claims~\ref{cla::o3staP} (i) and~\ref{cla::QAsmall}. Then,
\begin{displaymath}
|H \sm H_\pi| = \sum_{i=1}^3 |H_i| \le \left(\frac{1}{6} + 6\beta^{1/2} + 54\beta\right) tn^2 + |M|(2t+16) \beta^{1/2} n^2 . \qedhere
\end{displaymath}
\end{proof}

\begin{claim} \label{cla::lowbHpipex}
We have
$$
|\bar{H}_\pi| + \sum_{i=1}^3 \ex(|V_i|,\ftp)
\ge 
|M| \left( \frac{1}{9} - (2t+7) \beta^{1/2} \right)n^2
+\left( \frac{1}{18} - 2\beta^{1/2} \right) tn^2.
$$
\end{claim}
\begin{proof}
For $\{i,j,k\} = [3]$ and $x_1x_2 \in M_i$, by Lemma~\ref{lem::Qn29edges}, we have
$$
d_{V_j \sm E_j, V_k \sm E_k}(x_1) + d_{V_j \sm E_j, V_k \sm E_k}(x_2) 
\le |V_j \sm E_j||V_k \sm E_k| + (t+3)n
,
$$
so there are at least 
$$
|M_i| \left(2|V_j \sm E_j||V_k \sm E_k| - \left(|V_j \sm E_j||V_k \sm E_k| + (t+3)n\right)  \right)
$$
missing crossing hyperedges containing exactly one vertex in each of $E_i$, $V_j \sm E_j$, and $V_k \sm E_k$.
By Claims~\ref{cla::o3staP} (i) and~\ref{cla::QAsmall}, we have
$$
|\bar{H}_\pi|
\ge |M| \left(\left( \left(\frac{1}{3}-3\beta^{1/2} -2(t+2) \beta^{1/2} \right)n\right)^2 - (t+3)n\right)
\ge |M| \left( \frac{1}{9} - (2t+7) \beta^{1/2} \right)n^2
.
$$
By Claim~\ref{cla::o3staP} (i) and the lower bounds on $\ex(n,\ftp)$ in Proposition~\ref{pro::SteineralmostTuran3}, we have
\begin{align}
\label{ine:exF2t}
\sum_{i=1}^3 \ex(|V_i|,\ftp) \ge t\b{\left\lfloor \left(\frac{1}{3}- 3\beta^{1/2} \right)n\right\rfloor}{2} - \frac{n}{3} - \frac{t^2}{2}-9t \ge \left( \frac{1}{18} - 2\beta^{1/2} \right) tn^2.
\end{align}
Summing them up, we get the lower bound desired.
\end{proof}

\begin{claim} \label{cla::sizeOfM}
We have $|M| \le t$.
\end{claim}
\begin{proof}
Assume $|M|\geq t+1$. We get $|H \sm H_\pi| < |\bar{H}_\pi| + \sum_{i=1}^3 \ex(|V_i|,\ftp)$ by Claims~\ref{cla::uppHsmHpi} and~\ref{cla::lowbHpipex}. This implies $|H| < \sHop < \sHo$ by Lemma~\ref{lem::balLar}, a contradiction.
\end{proof}

Finally, Theorem~\ref{thm::UpperHpi} follows from a short calculation. Since $|H|\geq \sHop$, we have $|\bar{H}_\pi| \le |H \sm H_\pi| - \sum_{i=1}^3 \ex(|V_i|,\ftp)$ and thus by Claims~\ref{cla::uppHsmHpi}, \ref{cla::sizeOfM}, and \eqref{ine:exF2t}, we get 
\begin{align*}
|\bar{H}_\pi| &\leq \left(\frac{1}{6} + 6\beta^{1/2} + 54\beta\right) tn^2 + (2t+16) \beta^{1/2} tn^2- \left(\frac{1}{18} - 2\beta^{1/2}\right) tn^2 \\
&\leq   \left( \frac{1}{9} + (2t+25) \beta^{1/2} \right)tn^2 = \left( \frac{1}{9} + \eps \right)tn^2 
.
\qedhere
\end{align*}
\end{proof}

\begin{remark}
For $t=0$, Theorem~\ref{thm::UpperHpi} states that an extremal $F_5$-free $n$-vertex $3$-graph is complete tripartite for sufficiently large $n$. This proves Theorem~\ref{thm::Main} in the case $F_5^0 = F_5$.
\end{remark}

\subsection{Typical and non-typical vertices} \label{subsec::typ}
For every integer $t \ge 0$, let $c = c(t) \ce (t+100)^{-1000}$. For a $3$-graph $H$ on vertex set $[n]$ and a partition $\pi = (V_1,V_2,V_3)$ of $[n]$, define 
\begin{gather*}
B_i \ce \{x \in V_i: d_{j,k}(x) \le cn^2\}, \quad
A_i \ce V_i \sm B_i, \quad \textrm{for $\{i,j,k\} = [3]$}, \\
B \ce B_1\cup B_2 \cup B_3, \quad \text{and} \quad A \ce A_1 \cup A_2 \cup A_3.
\end{gather*}
We call vertices in $A$ \emph{typical} and vertices in $B$ \emph{non-typical}. In Sections~\ref{subsubsec::Nontyp} and~\ref{subsubsec::Typ}, we prove properties of non-typical and typical vertices, respectively.

\subsubsection{Non-typical vertices} \label{subsubsec::Nontyp}

\begin{lemma} \label{lem::B}
Let $t \ge 0$, $n$ be sufficiently large, and $H$ be an $\ffp$-free $3$-graph on vertex set $[n]$ with $|H| \ge \sHo$ and $\delta(H) \ge \frac{1}{9}n^2-n$. Further, let $\pi= (V_1,V_2,V_3)$ be a partition of $[n]$ maximizing $|H_\pi|$ and $c,B$ be defined as above. The following statements hold.
\begin{enumerate}[label=(\roman*)]
    \item $|B| \le t$.
    \item For every vertex $v \in B$, we have $d_{1,2}(v), d_{1,3}(v), d_{2,3}(v) \le cn^2$.
    \item For every pair of distinct vertices $v, v' \in B$, we have $d(v,v') \le t$.
\end{enumerate}
\end{lemma}

\begin{proof}[Proof of (i)]
By Claim~\ref{cla::o3staP} (i), Theorem~\ref{thm::UpperHpi}, and the definition of $B_i$, we first have $|B_i| \le t$ for $i \in [3]$.
Let $\{i,j,k\} = [3]$. For every $x \in B_i$, the number of missing crossing hyperedges that contain $x$ and do not contain any vertex from $B_j \cup B_k$ is at least 
\begin{align*}
    (|V_j|-|B_j|)(|V_k|-|B_k|) - cn^2\geq \left(\frac{1}{3}-4c \right)^2n^2  - cn^2 \geq \left( \frac{1}{9}-4c\right)n^2.
\end{align*}
Thus, by Theorem~\ref{thm::UpperHpi}, we have
$
|B| \left(  \frac{1}{9}-4c\right)n^2 
\le |\bar{H}_\pi| \le 
\left(\frac{1}{9} + c\right) tn^2,
$
implying $|B| \le t$.
\end{proof}

\begin{proof}[Proof of (ii)]
Let $\{i,j,k\} = [3]$. By the definition of $B_i$, we have $d_{j,k}(x) \le cn^2$ for every $x \in B_i$. As $\pi$ maximizes $|H_\pi|$, we have $d_{i,j}(x), d_{i,k}(x) \le d_{j,k}(x) \le cn^2$.
\end{proof}

\begin{proof}[Proof of (iii)]
Assume for contradiction that there exist vertices $v,v' \in B$ such that $d(v,v') \ge t+1$. By Claim~\ref{cla::o3staP} (i) and Lemma~\ref{lem::Qn29edges}, we have
$$
\sum_{i=1}^3 \left( d_{i,i}(v) + d_{i,i}(v') \right) \le \sum_{i=1}^3 \left( \frac{1}{2}|V_i|^2 + (t+4)n \right) \le \frac{n^2}{5}.
$$
Thus, without loss of generality, we can assume that $\sum_{i=1}^3 d_{i,i}(v) \le \frac{1}{10}n^2$. By (ii), we have
$$
d(v) = \sum_{i=1}^3 d_{i,i}(v) + \sum_{1\le i < j \le 3} d_{i,j}(v) \le \frac{1}{10}n^2 + 3 cn^2 
< \frac{1}{9} n^2 - n,
$$
contradicting the minimum degree assumption.
\end{proof}

\subsubsection{Typical vertices} \label{subsubsec::Typ}
\begin{lemma} \label{lem::A}
Let $t \ge 0$, $n$ be sufficiently large, and $H$ be an $\ffp$-free $3$-graph on vertex set $[n]$ with $|H| \ge \sHo$ and
$\delta(H)\geq \frac{1}{9}n^2-n$. Further, let $\pi = (V_1,V_2,V_3)$ be a partition of $[n]$ maximizing $|H_\pi|$ and $c,A_1,A_2,A_3$ be defined as above. The following statements hold.
\begin{enumerate}[label=(\roman*)]
    \item Let $\{i,j,k\} = [3]$. For every vertex $x \in A_i$, we have
        \begin{enumerate}
            \item $d_{i,j}(x), d_{i,k}(x) \le 3cn^2$,
            \item $d_{i,i}(x) \le n^{3/2}$, 
            \item $d_{j,j}(x), d_{k,k}(x) \le c n^2$, and
            \item $d_{j,k}(x) \ge (\frac{1}{9} - 10 c)n^2$.
        \end{enumerate}
    \item For $i \in [3]$ and every pair of distinct vertices $x,x' \in A_i$, we have           $d(x,x') \le t$.
    \item For $i\neq j \in [3]$ and vertices $x\in A_i$, $y \in A_j$, if $d_{A_i}(x,y) + d_{A_j} (x,y) > 0$, then we have $d_{A_k}(x,y) \le 50c n$.
\end{enumerate}
\end{lemma}

\begin{proof}[Proof of (i)]
(a) follows from Claim~\ref{cla::o3staP} (iii). For (b), assume for contradiction that there exists vertex $x \in A_i$ such that $d_{i,i}(x) > n^{3/2}$.
Let $S_x \ce \{x' \in V_i: d_i(x,x') \ge t+1\}$. Then
$$
  n^{3/2} < d_{i,i}(x) \le \frac{1}{2}\left(|S_x| |V_i| + |V_i \sm S_x|\cdot t\right) \le \frac{1}{2}\left(|S_x|n + tn\right),
$$
implying $|S_x| \ge n^{1/2}$. For every $x' \in S_x$, by Lemma~\ref{lem::Qn29edges}, we have $d_{j,k}(x) + d_{j,k}(x') \le |V_j||V_k| + (t+3)n$. By the definition of $A_i$, we have $d_{j,k}(x) \ge cn^2$ and hence $d_{j,k}(x') \le |V_j||V_k| + (t+3)n - cn^2\leq |V_j||V_k|  - \frac{1}{2}cn^2$. Thus,
\begin{align*}
    |\bar{H}_\pi|= \sum_{x'\in V_i} \left(|V_j||V_k|-d_{j,k}(x')\right) \geq \sum_{x'\in S_x} \left(|V_j||V_k|-d_{j,k}(x')\right) \geq |S_x| \cdot  \frac{1}{2}cn^2\geq \frac{1}{2}cn^{\frac{5}{2}},
\end{align*}
contradicting Theorem~\ref{thm::UpperHpi}. Next, we prove (c).
For every $x \in V_i$, define $W_j \ce \{y \in V_j: d_j(x,y) > n^{3/4}\}$, $W_k \ce \{z \in V_k: d_k(x,z) > n^{3/4}\}$, $\widebar{W}_j \ce V_j \sm W_j$, and $\widebar{W}_k \ce V_k \sm W_k$. By the definitions of $W_j$ and $W_k$, we have
$$
d_{W_j,\widebar{W}_j}(x), d_{\widebar{W}_j,\widebar{W}_j}(x), d_{W_k,\widebar{W}_k}(x), d_{\widebar{W}_k,\widebar{W}_k}(x)
\le n^{3/4} \cdot n\leq n^{7/4}.
$$
Further, for every $y\in W_j$, we have $d_k(x,y) \le n^{3/4}$; otherwise either there exist $y'\in L_j(x,y)$ and $z\in L_k(x,y)$ satisfying $d_i(y',z)\geq t+2$, and we find a copy of $\ffp$ in $H$; or for every $y'\in L_j(x,y)$ and $z\in L_k(x,y)$, it holds that $d_i(y',z) < t+2$, in which case 
\bdm
 |\bar{H}_\pi| \ge d_j(x,y)d_k(x,y)(|V_i|-t-2)\geq n^{3/2}(|V_i|-t-2),
\edm
contradicting Theorem~\ref{thm::UpperHpi} by Claim~\ref{cla::o3staP} (i). We conclude $d_{W_j, W_k}(x),d_{W_j,\widebar{W}_k}(x) 
\le n^{7/4}$, and by a similar argument also $d_{\widebar{W}_j, W_k}(x)\le n^{7/4}$. Therefore,
\begin{equation}
d_{j,j}(x) + d_{j,k}(x) + d_{k,k}(x) \le \b{|W_j|}{2} + (|V_j| - |W_j|)(|V_k| - |W_k|) + \b{|W_k|}{2} + 7n^{7/4}. \label{equ::1}
\end{equation}
On the other hand, by (a), (b), and the minimum degree assumption, we have
\begin{equation}
d_{j,j}(x) + d_{j,k}(x) + d_{k,k}(x) \ge d(x) - d_{i,i}(x) - d_{i,j}(x) - d_{i,k}(x) \ge \left(\frac{1}{9} - 7 c \right)n^2. \label{equ::2}
\end{equation}
Combining \eqref{equ::1} and \eqref{equ::2} and using Claim~\ref{cla::o3staP} (i), we get 
$$
\frac{1}{2}|W_j|^2 + \left( \left(\frac{1}{3} + 6 c\right)n -|W_j| \right) \left( \left(\frac{1}{3} + 6 c\right)n -|W_k| \right) + \frac{1}{2}|W_k|^2
\ge
\left(\frac{1}{9} - 8 c \right)n^2,
$$
which implies $|W_j| + |W_k| \le 50 cn$ or $|W_j| + |W_k| \ge (\frac{2}{3}-38c)n$. If $|W_j| + |W_k| \ge (\frac{2}{3}-38c)n$, then by the definition of $A_i$, we have
\begin{align*}
cn^2 &< d_{j,k}(x) 
=
d_{\widebar{W}_j, \widebar{W}_k}(x) + d_{W_j, \widebar{W}_k}(x) + d_{\widebar{W}_j,W_k}(x) + d_{W_j,W_k}(x) \\
&\le 
\left( \left(\frac{1}{3} + 6 c\right)n -|W_j| \right) \left( \left(\frac{1}{3} + 6 c\right)n -|W_k| \right) + 3n^{7/4}\\
&\le
\left( 25cn \right)^2 + 3n^{7/4} < cn^2
,
\end{align*}
a contradiction. We conclude $|W_j|, |W_k| \le 50 cn$ and thus
\bdm
d_{j,j}(x) = d_{W_j,W_j}(x) + d_{W_j,\widebar{W}_j}(x) + d_{\widebar{W}_j,\widebar{W}_j}(x) \le \b{50cn}{2} + 2n^{7/4} \le cn^2.
\edm
Similarly, $d_{k,k}(x) \le cn^2$. Finally, (d) follows by the minimum degree assumption, (a), (b), and (c). 
\end{proof}

\begin{proof}[Proof of (ii)]
If $x,x' \in A_i$ are distinct vertices such that $d(x,x') \ge t+1$, then by Lemma~\ref{lem::Qn29edges}, $d_{j,k}(x) + d_{j,k}(x') \le |V_j||V_k| + (t+3)n \le \frac{1}{8} n^2$. This contradicts (d) in (i).
\end{proof}

\begin{proof}[Proof of (iii)]
Fix vertex $v \in A_i \cup A_j$ such that $xyv \in H$. Without loss of generality, we assume $v \in A_i$. By (i) and Lemma~\ref{lem::B} (i), we have 
\begin{align}
\label{ineq:dakxy1}
d_{A_j,A_k} (v) \ge d_{j,k}(v) - d_{B_j, V_k}(v) - d_{V_j, B_k}(x) \ge  d_{j,k}(v) - t|V_i| - t|V_j| \ge \left(\frac{1}{9} - 11 c\right)n^2.
\end{align}
For every $z \in L_{A_k} (x,y)$, we have $d_{A_j}(v,z)\le t+1$; otherwise there exist $y_1,\ldots, y_{t+1}\in A_j \setminus \{y\}$ and then hyperedges $\{xyv,xyz,vzy_1,\ldots, vzy_{t+1}\}$ form a copy of $\ffp$ in $H$. Therefore,
\begin{align}
\label{ineq:dakxy2}
d_{A_j,A_k} (v) &\le (t+1) d_{A_k}(x,y) + |A_j| \left(|A_k| - d_{A_k}(x,y)\right).
\end{align}
Combining \eqref{ineq:dakxy1} with \eqref{ineq:dakxy2} and Claim~\ref{cla::o3staP} (i), we get $d_{A_k}(x,y) \le 50c n$.
\end{proof}

\subsection{All vertices are typical} \label{subsec::cle}
\begin{theorem} \label{thm::Bempty}
Let $t \ge 0$, $n$ be sufficiently large, and $H$ be an $\ffp$-free $3$-graph on vertex set $[n]$ with $|H| \ge \sHo$ and $\delta(H) \geq \frac{1}{9}n^2-n$. Further, let $\pi = (V_1,V_2,V_3)$ be a partition of $[n]$ maximizing $|H_\pi|$, and let $B$ be defined as in Section~\ref{subsec::typ}. Then, $|B| = 0$.
\end{theorem}
Let $c$, $A_1$, $A_2$, $A_3$, $A$, $B_1$, $B_2$, and $B_3$ be defined as in Section~\ref{subsec::typ}. We write $\varpi$ for the $3$-partition with parts $(A_1,A_2,A_3)$, and let $A_4:= B$. For $i,j,k\in [4]$, define $H_{ijk}$ to be the set of hyperedges $xyz\in H$ where $x \in A_i$, $y \in A_j$, and $z \in A_k$. Further, define $H_{bad}\ce \bigcup_{i,j\in [3], \ i \neq j } H_{iij}$.
Then,
\begin{equation}
|H| = |H_{123}| + \sum_{i=1}^3 |H_{iii}| + |H_{bad}|  +  \sum_{1\le i \leq j \le 3} |H_{ij4}| + \sum_{i=1}^4|H_{i44}|. \label{equ::H}
\end{equation}
Recall that we defined $H_\varpi = H_{123}$ and $\bar{H}_\varpi = K_\varpi \sm H_\varpi$ in Section~\ref{subsec::Not}. Note that $H_\varpi \subseteq H_\pi$ and $\bar{H}_\varpi \subseteq \bar{H}_\pi$. For $i \in [3]$, by Claim~\ref{cla::o3staP} (i) and Lemma~\ref{lem::B} (i), we have 
\begin{align}
\label{ineq: classsizes}
\left( \frac{1}{3} - 3c \right) n -t \le \left( \frac{1}{3} - 3c \right) n - |B| \leq |A_i|  \leq |V_i| \le  \left(\frac{1}{3}+6c \right)n.
\end{align}
For $i\neq j \in [3]$, let $G_{ij}$ be an auxiliary bipartite graph on $A_i \cup A_j$, where there is an edge between $x \in A_i$ and $y\in A_j$ iff there exists a vertex $v \in (A_i \cup A_j)$ such that $xyv \in H$. Denote by $d_{G_{ij}}(x)$ the degree of vertex $x$ in $G_{ij}$. 

\begin{proof}[Proof of Theorem~\ref{thm::Bempty}]
The proof is split into several claims. 

\begin{claim} \label{cla::Hbad}
 $|H_{bad}| \le 12tc^{1/2}n^2$.
\end{claim}
\begin{proof}
We bound $|H_{ijj}|$ for every $i\neq j \in [3]$.
For every $xy \in G_{ij}$, by Lemma~\ref{lem::A} (iii), we have $d_{A_k}(x,y) \le 50c n$. By Theorem~\ref{thm::UpperHpi}, we have
$$
|G_{ij}| \left(|A_k| - 50 c\right) \le |\bar{H}_\varpi| \le |\bar{H}_\pi| \le \left(\frac{1}{9} + c\right) tn^2,
$$
implying $|G_{ij}| \le tn$ by \eqref{ineq: classsizes}. We partition $A_i=S\cup T$, where $S$ is the set of vertices $x \in A_i$ satisfying $d_{G_{ij}}(x) > c^{1/2}n$ and $T:=A_i \setminus S$. 
Since $|G_{ij}| \le tn$, we have $|S| \le tc^{-1/2}$. By Lemma~\ref{lem::A} (i), we have $d_{A_j,A_j}(x) \le d_{j,j}(x) \le cn^2$ for every $x \in S$. Therefore, the number of hyperedges in $H_{ijj}$ containing a vertex in $S$ is at most $tc^{1/2} n^2$. The number of hyperedges in $H_{ijj}$ containing a vertex in $T$ is at most
    $$
    \sum_{x \in T} \binom{d_{G_{ij}}(x)}{2}\leq \sum_{x \in T} d_{G_{ij}}^2(x) 
    \le c^{1/2}n \sum_{x \in T} d_{G_{ij}}(x)  \leq c^{1/2}n |G_{ij}|\leq  t c^{1/2}n^2.
    $$
We conclude $|H_{ijj}| \le 2t c^{1/2} n^2$ and thus $|H_{bad}| \le 12t c^{1/2} n^2$.
\end{proof}

\begin{claim} \label{cla::Bfull}
For every vertex $v \in B$, we have $d_{A_i,A_i}(v) \ge \b{|A_i|}{2} - c^{1/3} n^2$ for $i \in [3]$.
\end{claim}
\begin{proof}
Assume that there exist vertex $v \in B$ and $k \in [3]$ such that $d_{A_{k},A_{k}}(v) < \b{|A_{k}|}{2} - c^{1/3} n^2$. Then, for every $i\in [3]$, 
\begin{equation*}
|H_{ii4}| \le |B|\b{|A_{i}|}{2} - \mathbbm{1}_{\{i=k\}} c^{1/3} n^2 \quad \text{and} \quad 3 |H_{iii}| + |H_{ii4}|  \le t \b{|A_i|}{2},
\end{equation*}
where the second inequality holds by Lemma~\ref{lem::A} (ii). Combining these two inequalities, we get
\begin{align}
\label{ineq:coll4}
\sum_{i=1}^3( |H_{iii}| + |H_{ii4}|)\le \left(\frac{t}{3} + \frac{2|B|}{3}\right) \left(\sum_{i=1}^3 \b{|A_{i}|}{2}\right) -\frac{2}{3}c^{1/3}n^2 \leq \frac{t}{18}n^2+\frac{|B|}{9}n^2-\frac{1}{2}c^{1/3}n^2,
\end{align}
where the last inequality holds by \eqref{ineq: classsizes}.
Now we have
\begin{gather}
\label{ineq:coll1}
|H_{123}|\le |A_1||A_2||A_3| \le \left(\frac{n-|B|}{3}\right)^3 = \frac{1}{27}n^3 - \frac{1}{9}|B|n^2 + \frac{1}{9} |B|^2n - \frac{1}{27} |B|^3, \\
\label{ineq:coll2}
 \sum_{1\le i < j \le 3} |H_{ij4}|\le |B| \cdot 3cn^2 \le 3tcn^2 \quad \text{by Lemma~\ref{lem::B} (i) (ii),} \quad \text{and} \\
 \label{ineq:coll3}
 \sum_{i=1}^4 |H_{i44}| \le \b{|B|}{2} t \le t^3 \quad \text{by Lemma~\ref{lem::B} (i) (iii).}
\end{gather}
Combining \eqref{equ::H}, \eqref{ineq:coll4}, \eqref{ineq:coll1}, \eqref{ineq:coll2}, \eqref{ineq:coll3}, and Claim~\ref{cla::Hbad}, we get 
$$
|H| \le \frac{1}{27}n^3 + \left(\frac{1}{18}t - \frac{1}{3} c^{1/3}\right)n^2 < \sHo,
$$ 
a contradiction.
\end{proof}

\begin{claim}
\label{cla:S_vsmall}
For vertex $v\in B$ and $i\in [3]$, let $S_{v,i} \ce \{x\in A_i : d_{A_i}(v,x) \leq t+2\}$. Then $|S_{v,i}|\leq 10c^{1/3}n$.  
\end{claim}
\begin{proof}
By the definition of $S_{v,i}$ and Claim~\ref{cla::Bfull},
\bdm
\binom{|A_i|-|S_{v,i}|}{2} + |S_{v,i}|(t+2)\ge  d_{A_i,A_i}(v) \ge \b{|A_i|}{2}-c^{1/3}n^2,
\edm
so $|S_{v,i}| \le 10c^{1/3}n$, where we also use \eqref{ineq: classsizes}.
\end{proof}

\begin{claim} \label{cla::BfHii}
If $|B| \ge 1$, then $|H_{iii}| \le (t+1)c^{1/3}n^2$ for $i\in [3]$.
\end{claim}
\begin{proof}
Let $v \in B$ and $i\in [3]$. 
By Claim~\ref{cla:S_vsmall}, $|S_{v,i}|\leq 10c^{1/3}n$. 
For every $xx'x'' \in H_{iii}$ where $x'' \in A_i\setminus S_{v,i}$, we have $xx' \notin L_{A_i,A_i}(v)$; otherwise there exist distinct vertices $x_1,\ldots, x_{t+1}\in L_{A_i}(v,x'')\setminus \{x,x'\}$, and the hyperedges $\{xx'v,xx'x''$, $vx''x_1,\ldots, vx''x_{t+1}\}$ form an $\ffp$ in $H$, a contradiction. 
By Lemma~\ref{lem::A} (ii), every pair in $A_i$ has codegree at most $t$. 
By Claim~\ref{cla::Bfull}, there are at most $\left(\b{|A_i|}{2} - d_{A_i,A_i}(v)\right) t \le tc^{1/3}n^2$ hyperedges in $H_{iii}$ containing at least one vertex in $A_i\setminus S_{v,i}$. 
There are at most $\b{|S_{v,i}|}{2} t \le c^{1/3}n^2$ hyperedges in $H_{iii}$ containing three vertices in $S_{v,i}$. Therefore, $|H_{iii}| \le (t+1) c^{1/3} n^2$.
\end{proof}

\begin{claim} \label{claim::B1t1}
$|B| \notin [1,t-1]$.
\end{claim}
\begin{proof}
Assume $|B| \in [1,t-1]$. By Claim~\ref{cla::BfHii}, $\sum_{i=1}^3 |H_{iii}| \leq 3(t+1)c^{1/3}n^2$. By \eqref{ineq: classsizes}, $\sum_{i=1}^3 |H_{ii4}|\leq |B| \sum_{i=1}^3 \b{|A_i|}{2} \le |B|(\frac{1}{6}+7c)n^2$. Then, by \eqref{equ::H}, \eqref{ineq:coll1}, \eqref{ineq:coll2}, \eqref{ineq:coll3}, and Claim~\ref{cla::Hbad}, 
we get
\begin{align*}
    |H|\le \frac{n^3}{27} +\frac{|B|n^2}{18} + 0.01n^2 < \sHo,
\end{align*}
a contradiction. 
\end{proof}

\begin{claim} \label{cla::Bt}
$|B|\neq t$.
\end{claim}
\begin{proof}
Assume $|B|=t$. By Lemma~\ref{lem::A} (ii), we have
\begin{align}
\label{ieq: badt}
\sum_{i=1}^3 |H_{iii}|+  |H_{bad}|+ \sum_{i=1}^3 |H_{ii4}| 
\le t \sum_{i=1}^3 \b{|A_i|}{2}.
\end{align}
If $ \sum_{1\le i < j \le 3} |H_{ij4}| \leq \frac{n}{11}$, then by \eqref{equ::H}, \eqref{ineq:coll3}, \eqref{ieq: badt}, and Lemma~\ref{lem::HoHt}, 
\bdm
|H| \le |A_1||A_2||A_3| + t \sum_{i=1}^3 \b{|A_i|}{2} + \frac{n}{11} + t^3 \le \sHtp + \frac{n}{10} < \sHo,
\edm
a contradiction. Hence, $ \sum_{1\le i < j \le 3} |H_{ij4}| > \frac{n}{11}$, so there exist $v\in B$ and $i\neq j \in \{1,2,3\}$ such that $d_{A_i,A_j}(v) \ge \frac{n}{33t}$. For every $x\in A_i,y\in A_j$ with $xyv\in H$ and every $z \in L_{A_k}(x,y)$, we have $d_{A_k}(v,z)< t+1$; otherwise there exist distinct vertices $z_1, \ldots, z_{t+1} \in L_{A_k}(v,z)$, and then hyperedges $\{xyv, xyz, vzz_1,\ldots, vzz_{t+1}\}$ form a copy of $\ffp$ in $H$, a contradiction. By Claim~\ref{cla:S_vsmall}, $d_{A_k}(x,y)\leq |S_{v,k}|\le 10c^{1/3}n$. Thus,
$|\bar{H}_\varpi| \ge \frac{n}{33t} (|A_k| - 10c^{1/3}n)$. 

Let $\{i,j,k\} = [3]$. For every vertex $v \in B_i$, by Lemma~\ref{lem::B} (ii), there are at least $|A_j||A_k| - cn^2$ missing crossing hyperedges in $\bar{H}_\pi$ containing $v$ and two vertices in $A$. By \eqref{ineq: classsizes}, we conclude
\bdm
|\bar{H}_\pi| \ge |\bar{H}_\varpi| + t \left(\left( \frac{1}{3} n - 3cn -t \right)^2 - cn^2\right)
> \left(\frac{t}{9} + \frac{1}{100t}\right) n^2 > \left(\frac{1}{9} + c\right) tn^2,
\edm
contradicting Theorem~\ref{thm::UpperHpi}.
\end{proof}

By Lemma~\ref{lem::B} (i) and Claims~\ref{claim::B1t1} and~\ref{cla::Bt}, $|B| = 0$.
\end{proof}

\subsection{Proof of Theorem~\ref{thm::Main} with a minimum degree assumption}

\label{subsec::proMin}
\begin{theorem} \label{thm::MainMin}
For every $t \ge 0$ and sufficiently large $n$, let $H$ be a $3$-graph on vertex set $[n]$ with maximum number of hyperedges among the $n$-vertex $\ffp$-free $3$-graphs satisfying $\delta(H)\geq \frac{1}{9}n^2 - n$. Then, $H \in \fHo$.
\end{theorem}

\begin{proof}
Let $\pi = (V_1,V_2,V_3)$ be a partition of $[n]$ maximizing $|H_\pi|$, and let $c$, $A_1$, $A_2$, $A_3$, $B$ be defined as in Section~\ref{subsec::typ}. Note that $\delta(\Ho) \ge \delta(S(n)) \ge \frac{1}{9}n^2 -n$, so $|H| \ge \sHo$. By Theorem~\ref{thm::Bempty}, $|B| = 0$, and hence $V_i = A_i$ for $i\in [3]$. We will prove $H_{bad} = \es$. Without loss of generality, we assume $|H_{112}| + |H_{221}| \ge |H_{113}| + |H_{331}| \ge |H_{223}| + |H_{332}|$.

\begin{claim} \label{cla::oneThenmany}
If $|H_{bad}| \ge 1$, then $|G_{12}| > \frac{1}{4}n^{1/2}$.
\end{claim}
\begin{proof}
By Lemma~\ref{lem::balLar}, we have $|H| \ge \sHo \ge \sHop$, and hence,
\begin{equation} 
 \sum_{i=1}^3 |H_{iii}| + |H_{bad}|
 = |H \sm H_\pi| 
 \ge |\bar{H}_\pi| + \sum_{i=1}^3 \ex(|V_i|,\ftp). \label{equ::B0}
\end{equation}
For $i\in [3]$, by Lemma~\ref{lem::A} (ii), every pair of distinct vertices in $A_i$ has codegree at most $t$, so $|H_{iii}| \le \ex(|V_i|,\ftp)$. Hence, we have $|H_{bad}| \ge |\bar{H}_\pi|$. By \eqref{ineq: classsizes} and Lemma~\ref{lem::A} (iii), we have 
\begin{equation}
|\bar{H}_\pi| \ge  |G_{12}|\cdot \left(\frac{1}{3} - 3c -50c \right)n. \label{equ::OneThird}
\end{equation}
If $|H_{bad}| \ge 1$, then $|H_{112}| + |H_{221}| \ge 1$ and hence $|G_{12}| \ge 1$ by the definition of $G_{12}$.
we have 
$$
2\b{|G_{12}|}{2} \ge |H_{112}| + |H_{221}| \ge
\frac{1}{3} |H_{bad}| \ge \frac{1}{3} |\bar{H}_\pi| \ge
\frac{1}{3} \left(\frac{1}{3} - 53c \right)n,
$$ 
implying $|G_{12}| > \frac{1}{4}n^{1/2}$.
\end{proof}

\begin{claim} \label{cla::B0}
If  $|H_{bad}| \ge 1$, then there exist $x_0\in V_1$, $y_0 \in V_2$ such that $d_{G_{12}}(x_0) + d_{G_{12}}(y_0) \ge (\frac{1}{3} - 54c)n$.
\end{claim}

\begin{proof}
By Lemma~\ref{lem::A} (ii),
\begin{equation} 
3|H_{iii}| + |H_{iij}| + |H_{iik}| \le t\b{|V_i|}{2} \quad \textrm{for $\{i,j,k\} = [3]$.} \label{equ::B01}
\end{equation}
By the definition of $G_{12}$,
\begin{equation}
|H_{iij}|+|H_{jji}| \le |H_{112}|+|H_{221}| \le  \sum_{x \in V_1}\b{d_{G_{12}}(x)}{2}+ \sum_{y \in V_2}\b{d_{G_{12}}(y)}{2} \quad \textrm{for $i\neq j \in [3]$.} \label{equ::B02}
\end{equation}
Combining \eqref{equ::B01} and \eqref{equ::B02}, we have
\begin{align}
\sum_{i=1}^3 |H_{iii}| + |H_{bad}| 
&\le \frac{1}{3} t \sum_{i=1}^3 \b{|V_i|}{2} + \frac{2}{3} \cdot 3 \left(\sum_{x \in V_1}\b{d_{G_{12}}(x)}{2}+ \sum_{y \in V_2}\b{d_{G_{12}}(y)}{2}\right) \notag \\
&\le \sum_{i=1}^3 \ex(|V_i|, \ftp)+ \frac{n}{3} + \frac{t^2}{2} + 9t + \sum_{x \in V_1} d_{G_{12}}^2(x) + \sum_{y \in V_2} d_{G_{12}}^2(y) \label{equ::last}
,
\end{align}
where for the last inequality, we use the lower bound on $\ex(m,\ftp)$ from Proposition~\ref{pro::SteineralmostTuran3}.
By \eqref{equ::B0} and \eqref{equ::last},
\begin{equation} \label{equ::346}
\sum_{\{x,y\} \in G_{12}} (d_{G_{12}}(x) + d_{G_{12}}(y)) = \sum_{x \in V_1}d_{G_{12}}^2(x)+ \sum_{y \in V_2}d_{G_{12}}^2(y)
\ge |\bar{H}_\pi| -\frac{n}{3} - \frac{t^2}{2} - 9 t. 
\end{equation}
Using \eqref{equ::OneThird} and \eqref{equ::346}, we get
\begin{equation*}
\sum_{\{x,y\} \in G_{12}} (d_{G_{12}}(x) + d_{G_{12}}(y))  
\ge |G_{12}|\left(\frac{1}{3} - 53c\right)n -\frac{n}{3} - \frac{t^2}{2} - 9 t \ge |G_{12}| \left(\frac{1}{3} -  54c\right) n,
\end{equation*}
where the last inequality holds by Claim~\ref{cla::oneThenmany}. By the pigeon-hole principle, there exist $x_0\in V_1$, $y_0 \in V_2$ such that $d_{G_{12}}(x_0) + d_{G_{12}}(y_0) \ge (\frac{1}{3} - 54c)n$.
\end{proof}

\begin{claim}
\label{Hbadempty}
$H_{bad} = \es$.
\end{claim}

\begin{proof}
Assume for contradiction that $|H_{bad}| \ge 1$. By Claim~\ref{cla::B0}, we can assume that there exists $x_0 \in V_1$ such that $d_{G_{12}}(x_0) \ge (\frac{1}{6} -27c)n$. By \eqref{ineq: classsizes} and Lemma~\ref{lem::A} (iii), $x_0$ is contained in at least $(\frac{1}{6} - 27c)n \cdot (\frac{1}{3} - 53c)n > \frac{1}{20} n^2$ missing crossing hyperedges, contradicting Lemma~\ref{lem::A} (i). Hence, $H_{bad} = \es$.
\end{proof}

By Claim~\ref{Hbadempty} and Lemmas~\ref{lem::A} (ii) and~\ref{lem::balLar}, we get $|H| \le \sHop \leq \sHo $, where equality holds iff $H \in \fHo$.
\end{proof}

\subsection{Proof of Theorem \ref{thm::Main}} \label{subsec::proMai}
\begin{proof}[Theorem~\ref{thm::MainMin} $\Rightarrow$ Theorem~\ref{thm::Main}]
Assume that the statement of Theorem~\ref{thm::Main} does not hold. Then, for some $t \ge 0$, there exist infinitely many positive integers $n$ such that there exists an $n$-vertex $\ffp$-free hypergraph $H_n \notin \fHo$ with $|H_n| \ge \sHo$. Choose a sufficiently large $n$ among them. If $H_n$ contains a vertex $v$ with degree less than $\frac{1}{9}n^2 - n$, let $H_{n-1} = H_n - v$. Repeat this process until there is no such low-degree vertex. We obtain a sequence of hypergraphs $H_n,H_{n-1},\ldots$, where $H_m$ is obtained from $H_{m+1}$ by deleting a vertex in $H_{m+1}$ with degree less than $\frac{1}{9}(m+1)^2 - (m+1)$. Note that we have $|H_{n-l}| \ge \sHo[n-l] + l$, since $\sHo[m] - \sHo[m-1]>\frac{1}{9}m^2 - m$ for sufficiently large $m$ by Lemma~\ref{lem::LowerHo}. This sequence is finite and ends with more than $n_0 = \frac{1}{100}n^{1/3}$ vertices remaining; otherwise there would be at least $n - n_0 > \b{n_0}{3}$ hyperedges in $H_{n_0}$. When the process ends, we reach an $\hat{n}$-vertex $\ffp$-free $3$-graph $H_{\hat{n}}$, where $\hat{n} > n_0$, $|H_{\hat{n}}| > \sHo[\hat{n}]$, and the minimum degree of $H_{\hat{n}}$ is at least $\frac{1}{9}\hat{n}^2 - \hat{n}$.

Now taking $n$ to infinity, we have that there are infinitely many integers $m$ such that there exists an $m$-vertex $\ffp$-free $3$-graph $H_m$ with $|H_m| > \sHo[m]$ and $\delta(H_m) \ge \frac{1}{9}m^2 - m$, contradicting Theorem~\ref{thm::MainMin}.
\end{proof}

\section{Concluding Remarks} \label{sec::Con}
Recall that $S(n)$ denotes the complete tripartite $3$-graph on vertex set $[n]$ and $s(n)=|S(n)|$. An interesting question is to determine all $3$-graphs $H$ with $\ex(n,H) = s(n)$ for sufficiently large $n$. However, this seems to be rather difficult. A more approachable question is the following one. Recall that we write $F_5[t]$ for the $t$-blow-up of $F_5$. For $t \ge 1$, let $\mH^t$ be the family of $3$-graphs $F$ such that $F_5 \subseteq F \subseteq F_5[t]$.

\begin{question} \label{que::last}
For which $F \in \mH^t$ does $\ex(n,F) = s(n)$ hold for sufficiently large $n$?
\end{question}

Let $S'(m)$ be the $3$-graph obtained from $S(m)$ by adding a hyperedge containing three vertices in $V_1$, see Figure~\ref{fig::S'm}. Let $S''(m)$ be the $3$-graph obtained from $S(m)$ by adding a hyperedge containing two vertices in $V_1$ and one vertex in $V_2$, see Figure~\ref{fig::S''m}. Inspired by Simonovits's result~\cite{simonovits1968method} about critical graphs, one may have the natural guess that $\ex(n,F) = s(n)$ for sufficiently large $n$ iff $F \subseteq S'(m)$ and $F \subseteq S''(m)$ for some $m \ge 1$. This assumption is necessary, since $|S'(n)|=|S''(n)| = s(n)+1$. However, it is not sufficient, as shown by the following example.

\begin{figure}[ht]

\tikzstyle{every node}=[circle, draw, fill=black!80, inner sep=0pt, minimum width=2.4pt]
    \begin{minipage}{0.48\textwidth}
	\centering
	\begin{tikzpicture}[scale = 1.4]
	    \draw[color=white, opacity = .0, use as bounding box] (0,-1.3) rectangle (2,1.5);
	
	    \node (x) at (1, 0.7) [label=right:] {};  
	    \node (y) at (0.25, -0.6) [label=below:] {};
	    \node (z) at (1.75, -0.6) [label=below:] {};
	    
	    \node (x1) at (0.9, 1) [label={[label distance=0.01cm]45:$x_1$}] {};
	    \node (x2) at (0.67, 0.8)  [label=left:$x_2$] {};
	    \node (x3) at (0.82, 0.7)  [label={[label distance=0.02cm]265:$x_3$}] {};

	    \node at (1.9, 0.8) [fill=black!0,draw=black!0] {$V_1$};
	    \node at (-0.2, 0.2) [fill=black!0,draw=black!0] {$V_2$};
	    \node at (2.2, 0.2) [fill=black!0,draw=black!0] {$V_3$};
	
		\draw [very thick] (1,0.8) ellipse (0.7cm and 0.5cm);
		\draw [very thick] (0,-0.5) ellipse (0.7cm and 0.5cm);
		\draw [very thick] (2,-0.5) ellipse (0.7cm and 0.5cm);
		
        \draw[thick] (x) -- (y) -- (z) -- (x);
        \draw[thick] (x1) -- (x2) -- (x3) -- (x1);
		
    \end{tikzpicture}
    \caption{Hypergraph $S'(m)$.} \label{fig::S'm}
    \end{minipage}
    \begin{minipage}{0.48\textwidth}
	\centering
	\begin{tikzpicture}[scale = 1.4]
	    \draw[color=white, opacity = .0, use as bounding box] (0,-1.3) rectangle (2,1.5);
	
	    \node (x) at (1, 0.7) [label=right:] {};  
	    \node (y) at (0.25, -0.6) [label=below:] {};
	    \node (z) at (1.75, -0.6) [label=below:] {};
	    
	    \node (x1) at (0.6, 0.8) [label=above:$x_1$] {};
	    \node (x2) at (0.75, 0.7) [label={[label distance=0.03cm]70:$x_2$}] {};
	    \node (y1) at (0.05, -0.35) [label=below:$y_1$] {};

	    \node at (1.9, 0.8) [fill=black!0,draw=black!0] {$V_1$};
	    \node at (-0.2, 0.2) [fill=black!0,draw=black!0] {$V_2$};
	    \node at (2.2, 0.2) [fill=black!0,draw=black!0] {$V_3$};
	
		\draw [very thick] (1,0.8) ellipse (0.7cm and 0.5cm);
		\draw [very thick] (0,-0.5) ellipse (0.7cm and 0.5cm);
		\draw [very thick] (2,-0.5) ellipse (0.7cm and 0.5cm);
		
        \draw[thick] (x) -- (y) -- (z) -- (x);
        \draw[thick] (x1) -- (x2) -- (y1) -- (x1);
		
    \end{tikzpicture}
    \caption{Hypergraph $S''(m)$.} \label{fig::S''m}
    \end{minipage}
\end{figure}

Let $\hat{F}$ be the $3$-graph with vertex set 
\bdm
\{1,2,3,4,5\}\cup \{c_{12},c_{23},c_{31}\} \cup \bigcup_{\substack{l=1,2\\i=1,2,3}} \{a_i^l, b_i^l\}
\edm
and hyperedge set
\bdm
\{123,124,345\}\cup \{c_{12}b_1^2a_2^2, c_{23}b_2^2a_3^2, c_{31}b_3^2a_1^2\} \cup \bigcup_{i=1,2,3} \{5a_i^1b_i^1,5a_i^1b_i^2,5a_i^2b_i^1,5a_i^2b_i^2\},
\edm
see Figure~\ref{fig::hF}.
Hence, the $3$-graph $\hat{F}$ contains $F_5$ and other hyperedges such that the link graph of vertex $5$ contains three extra disjoint complete bipartite graphs on $\{a_i^1, a_i^2\} \cup \{b_i^1, b_i^2\}$. Besides, there are three vertices $c_{12},c_{23},c_{31}$ connecting these three complete bipartite graphs. 

\tikzmath{
\xf5 = -2;
\yf5 = 1;
\drf5 = 1;
\dcf5 = 0.9;
\xk22 = \xf5 + 4;
\yk22 = \yf5+0.2;
\drk22 = 0.5;
\dck22 = 0.3;
\dbetk22 = 0.99;
\drboxk22 = 0.4;
\dcboxk22 = 0.18;
\yc1 = (\yk22-\dck22 + \yk22-\dbetk22) /2;
\yc2 = (\yk22-\dck22-\dbetk22 + \yk22-2*\dbetk22)/2;
\xc3 = \xk22 + 2;
\yc3 = (\yc1 + \yc2) /2;
} 

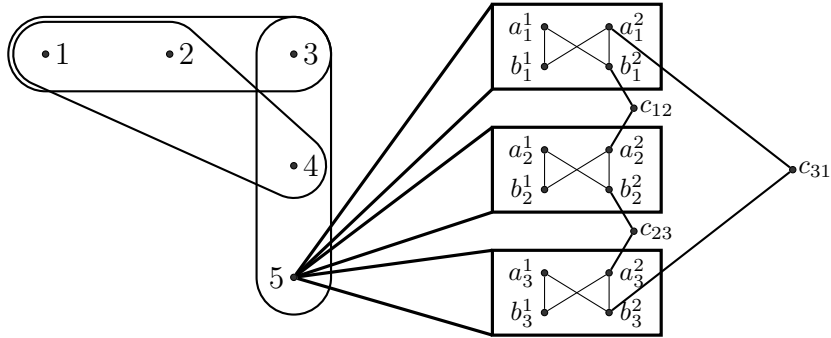
\begin{figure}[ht]
\tikzstyle{every node}=[circle, draw, fill=black!80, inner sep=0pt, minimum width=2.4pt]

	\centering
	\begin{tikzpicture}[scale = 1.65]
		\draw[color=white, opacity = .0, use as bounding box] (0,-1.3) rectangle (2,1.5);
		
		\node (1) at (\xf5, \yf5)  [label=right:$1$] {};  
		\node (2) at (\xf5+\drf5, \yf5)  [label=right:$2$] {};
		\node (3) at (\xf5+2*\drf5, \yf5)  [label=right:$3$] {};
		\node (4) at (\xf5+2*\drf5, \yf5-\dcf5)  [label=right:$4$] {};
		\node (5) at (\xf5+2*\drf5, \yf5-\dcf5-0.9) [label=left:$5$] {};
		
		\node (a11) at (\xk22,\yk22) [label = left:\footnotesize${a^1_1}$] {};
		\node (a21) at (\xk22+\drk22,\yk22) [label = right:\footnotesize$a^2_1$] {};
		\node (b11) at (\xk22,\yk22-\dck22) [label = left:\footnotesize$b^1_1$] {};
		\node (b21) at (\xk22+\drk22,\yk22-\dck22) [label = right:\footnotesize$b^2_1$] {};
		
		\node (c12)  at (\xk22+\drk22+0.2, \yc1) [label = right:\footnotesize$c_{12}$] {};
		
		\node (a12) at (\xk22,\yk22-\dbetk22) [label = left:\footnotesize$a^1_2$] {};
		\node (a22) at (\xk22+\drk22,\yk22-\dbetk22) [label = right:\footnotesize$a^2_2$] {};
		\node (b12) at (\xk22,\yk22-\dck22-\dbetk22) [label = left:\footnotesize$b^1_2$] {};
		\node (b22) at (\xk22+\drk22,\yk22-\dck22-\dbetk22) [label =right:\footnotesize$b^2_2$] {};
		
		\node (c23)  at (\xk22+\drk22+0.2, \yc2) [label = right:\footnotesize$c_{23}$] {};
		
		\node (a13) at (\xk22,\yk22-2*\dbetk22) [label = left:\footnotesize$a^1_3$] {};
		\node (a23) at (\xk22+\drk22,\yk22-2*\dbetk22) [label = right:\footnotesize$a^2_3$] {};
		\node (b13) at (\xk22,\yk22-\dck22-2*\dbetk22) [label = left:\footnotesize$b^1_3$] {};
		\node (b23) at (\xk22+\drk22,\yk22-\dck22-2*\dbetk22) [label =right:\footnotesize$b^2_3$] {};
		
		\node (c31) at (\xc3,\yc3) [label=right:\footnotesize$c_{31}$] {};
		
		\begin{pgfonlayer}{bg}
            \draw[thick] \hedgeiii{1}{2}{3}{3mm};
		    \draw[thick] \hedgeiii{1}{2}{4}{2.6mm};
		    \draw[thick] \hedgeiii{3}{4}{5}{3mm};
		    
		    \draw[very thick] (\xk22-\drboxk22, \yk22+\dcboxk22) rectangle (\xk22+\drk22+\drboxk22, \yk22 - \dck22 - \dcboxk22) {};
		    \draw[very thick] (\xk22-\drboxk22, \yk22-\dbetk22+\dcboxk22) rectangle (\xk22+\drk22+\drboxk22, \yk22-\dbetk22 - \dck22 - \dcboxk22) {};
		    \draw[very thick] (\xk22-\drboxk22, \yk22-2*\dbetk22+\dcboxk22) rectangle (\xk22+\drk22+\drboxk22, \yk22-2*\dbetk22 - \dck22 - \dcboxk22) {};
		    
		    \draw[very thick] (5) -- (\xk22-\drboxk22, \yk22+\dcboxk22);
		    \draw[very thick] (5) -- (\xk22-\drboxk22, \yk22-\dck22 - \dcboxk22);
		    \draw[very thick] (5) -- (\xk22-\drboxk22, \yk22-\dbetk22+\dcboxk22);
		    \draw[very thick] (5) -- (\xk22-\drboxk22, \yk22-\dck22-\dbetk22 - \dcboxk22);
		    \draw[very thick] (5) -- (\xk22-\drboxk22, \yk22-2*\dbetk22+\dcboxk22);
		    \draw[very thick] (5) -- (\xk22-\drboxk22, \yk22-\dck22-2*\dbetk22 - \dcboxk22);
		    
		    \draw (a11) -- (b11);
		    \draw (a11) -- (b21);
		    \draw (a21) -- (b11);
		    \draw (a21) -- (b21);
		    
		    \draw[thick] (b21) -- (c12) -- (a22);
		    
		    \draw (a12) -- (b12);
		    \draw (a12) -- (b22);
		    \draw (a22) -- (b12);
		    \draw (a22) -- (b22);
		    
		    \draw[thick] (b22) -- (c23) -- (a23);

		    \draw (a13) -- (b13);
		    \draw (a13) -- (b23);
		    \draw (a23) -- (b13);
		    \draw (a23) -- (b23);		    
		    
		    \draw[thick] (a21) -- (c31) -- (b23);
        \end{pgfonlayer}
    \end{tikzpicture}
    
    \caption{Hypergraph $\hat{F}$.} \label{fig::hF}
\end{figure}

\begin{claim}
We have $F_5 \subseteq \hat{F} \subseteq F_5[7]$, $\hat{F} \subseteq S'(30),S''(30)$, and $\ex(n,\hat{F})>s(n)$ for suffi\-ciently large $n$.
\end{claim}
\begin{proof}
The first two claims hold by the definition of $\hat{F}$. For the third claim, let $\hat{S}(n)$ be the $3$-graph obtained from $S(n-1)$ by adding a new vertex $v$ and adding hyperedges such that $L_{1,1}(v)$ and $L_{2,2}(v)$ are complete graphs $K_{|V_1|}$ and $K_{|V_2|}$ respectively, and $L_{3,3}(v)$ is an extremal $K_{2,2}$-free graph. We have $|\hat{S}(n)| > s(n)$ for sufficiently large $n$, since $\ex(n,K_{2,2}) \ge cn^{3/2}$ for some constant $c >0$ by Erd\H{o}s, R\'{e}nyi, and S\'{o}s~\cite{erdHos1966problem}. 
If there exists a copy of $\hat{F}$ in $\hat{S}(n)$, then the vertex in $\hat{S}(n)$ corresponding to vertex $5$ in $\hat{F}$ can only be the new vertex $v$, and then $\{a_i^1, a_i^2,b_i^1, b_i^2\}$ is in different parts of the partition of $\hat{S}(n)$ for different $i$. However, $L_{3,3}(v)$ is $K_{2,2}$-free, a contradiction.
\end{proof}

We have the following sufficient condition for Question~\ref{que::last}. Let $S^+(m)$ be the $3$-graph obtained from $S(m-1)$ by adding a new vertex $v$ and a hyperedge containing $v$ and two vertices in $V_1$, see Figure~\ref{fig::Sp}. Note that $S^+(m) \subseteq S'(m+3), S''(m+3)$. Using a cleaning method similar to the one in Section~\ref{sec::Pro}, we can prove the following claim; we omit its proof. 

\begin{claim}
For $3$-graph $F \in \mH^t$, if $F \subseteq S^+(m)$ for some $m \ge 1$, then $\ex(n,F) = s(n)$ for sufficiently large $n$.
\end{claim}
However, this is not a complete answer to Question~\ref{que::last}, due to the following example. Let $F_5'$ be the $3$-graph on vertex set $\{1,2,3,3',4,4',5\}$ with hyperedges $\{123,124,345, 3'4'5\}$, see Figure~\ref{fig::F'5}. 

\begin{figure}[H]

\tikzstyle{every node}=[circle, draw, fill=black!80, inner sep=0pt, minimum width=2.4pt]

    \begin{minipage}{0.48\textwidth}
	\centering
	\begin{tikzpicture}[scale = 1.4]
	    \draw[color=white, opacity = .0, use as bounding box] (0,-1.3) rectangle (2,1.5);
	
	    \node (x) at (1, 0.7) [label=right:] {};  
	    \node (y) at (0.25, -0.6) [label=below:] {};
	    \node (z) at (1.75, -0.6) [label=below:] {};
	    
	    \node (x1) at (0.6, 1) [label=right:$x_1$] {};
	    \node (x2) at (0.6, 0.8)  [label=right:$x_2$] {};
	    \node (v) at (0, 0.9)  [label=left:$v$] {};

	    \node at (1.9, 0.8) [fill=black!0,draw=black!0] {$V_1$};
	    \node at (-0.2, 0.2) [fill=black!0,draw=black!0] {$V_2$};
	    \node at (2.2, 0.2) [fill=black!0,draw=black!0] {$V_3$};
	
		\draw [very thick] (1,0.8) ellipse (0.7cm and 0.5cm);
		\draw [very thick] (0,-0.5) ellipse (0.7cm and 0.5cm);
		\draw [very thick] (2,-0.5) ellipse (0.7cm and 0.5cm);
		
        \draw[thick] (x) -- (y) -- (z) -- (x);
        \draw[thick] (x1) -- (x2) -- (v) -- (x1);
		
    \end{tikzpicture}
    \caption{Hypergraph $S^+(m+1)$.} \label{fig::Sp}
    \end{minipage}
	\begin{minipage}{0.48\textwidth}
	\centering
	\begin{tikzpicture}[scale = 1.4]
		\draw[color=white, opacity = .0, use as bounding box] (0,-1.3) rectangle (2,1.5);
		\node (1) at (0, 0.8)  [label=right:$1$] {};  
		\node (2) at (1, 0.8)  [label=right:$2$] {};
		\node (3) at (2, 0.8)  [label=right:$3$] {};
		\node (4) at (2, 0)  [label=right:$4$] {};
		\node (5) at (2, -0.8) [label=right:$5$] {};
		\node (3p) at (0, -0.8)  [label=right:$3'$] {};
		\node (4p) at (1, -0.8)  [label=right:$4'$] {};
		
		\begin{pgfonlayer}{bg}
            \draw[thick] \hedgeiii{1}{2}{3}{3mm};
		    \draw[thick] \hedgeiii{1}{2}{4}{2.6mm};
		    \draw[thick] \hedgeiii{3}{4}{5}{3mm};
		    \draw[thick] \hedgeiii{3p}{4p}{5}{2.7mm};
        \end{pgfonlayer}
    \end{tikzpicture}
    \caption{Hypergraph $F'_5$.} \label{fig::F'5}
    \end{minipage}
\end{figure}

\begin{claim}
We have $F_5 \subseteq F_5' \subseteq F_5[2]$, $F_5'$ is not a subhypergraph of $S^+(m)$ for any $m \ge 1$, and $\ex(n,F_5') = s(n)$ for every sufficiently large $n$. 
\end{claim}
\begin{proof}
The first claim holds by the definition of $F_5'$. 
Assume that $S^+(m)$ contains a copy of $F_5'$. Then, only the new vertex $v$ in $S^+(m)$ can be the vertex $5$ in $F_5'$. However, vertex $5$ has degree two in $F_5'$ while $v$ has degree one in $S^+(m)$, a contradiction.

Trivially, we have $\ex(n, F_5') \geq \ex(n, F_5) \geq s(n)$. For the other direction, using a similar argument as in Section~\ref{subsec::proMai}, we can assume that the extremal $F_5'$-free $3$-graph has minimum degree at least $\frac{1}{9}n^2-n$. For every $n$-vertex $3$-graph $H$ with $|H| > s(n)$, by Theorem~\ref{thm::Main}, $H$ contains a copy of $F_5$. Let this copy be $\{v_1v_2v_3, v_1v_2v_4, v_3v_4v_5\}$.
By the minimum degree assumption, $v_5$ is contained in a hyperedge $e$ that does contain any $v_i$ for $ i \in [4]$. Now, $\{v_1v_2v_3,v_1v_2v_4, v_3v_4v_5, e\}$ forms a copy of $F_5'$. Therefore, $\ex(n,F_5') = s(n)$ for every sufficiently large $n$.
\end{proof}

\paragraph*{Acknowledgments.} 
We are very grateful to
Xizhi Liu,
Dhruv Mubayi, and
Oleg Pikhurko
for many helpful comments on our work.
This project started as an IGL 2021 Summer program. The undergraduate students Junsheng Liu, Alexander Roe, Yuzhou Wang, and Zihan Zhou took part in the initial discussions. During Summer 2021, Alexander Roe was partially supported by RTG NSF grant DMS-1937241.

\bibliographystyle{abbrv}
\bibliography{expohyp}

\section*{Appendix}
\addtocounter{section}{1}
\setcounter{definition}{0}
In this appendix, we prove Proposition~\ref{pro::SteineralmostTuran3}.
The upper bound follows directly from the fact that an $\ftp$-free hypergraph has maximum codegree at most $t$ and the equality is achieved only for $(n,3,2,t)$-designs. Next, we prove the lower bounds. 

Let $G$ and $F$ be $r$-graphs. An $(F,t)$\emph{-design} of $G$ is a family $\mathcal{F}$ of distinct copies of $F$ in $G$ such that every hyperedge of $G$ is contained in exactly $t$ of these copies. The main tool for proving Proposition~\ref{pro::SteineralmostTuran3} is a result by Glock, K\"uhn, Lo, and Osthus~\cite{glock2016existence}, which gives sufficient conditions on $G$ to contain an $(F,t)$-design. Denote by $K_n$ the complete graph on $n$ vertices. Here, we will only state their result for the case we are interested in: $r=2$ and $F=K_3$. An $n$-vertex graph $G$ is $(c, h, p)$\emph{-typical} if for every set $A \subseteq V(G)$ with $|A| \leq h$, we have $(1-c)np^{|A|} \le |\bigcap_{v\in A} N(v)| \le (1 + c)np^{|A|}$, where $N(v)$ is the set of neighbors of $v$ in $G$. A graph $G$ is $(K_3,t)$\emph{-divisible} if $t|E(G)|$ is divisible by $3$ and $t |N(v)|$ is divisible by $2$ for every $v\in V(G)$.
\begin{theorem}[Glock, K\"uhn, Lo, and Osthus~\cite{glock2016existence}]
\label{GlockKuhndLoOsthusdesign}
For $h = 2812$ and every $c,p\in\left(0,1\right]$ with
$c\leq 0.9(p/2)^{h}/(36^{2}\cdot4^{36})$,
there exist $n_0$ and $\gamma > 0$ such that the following holds for all $n \geq n_0$.
 Let $t$ be a positive integer with $t \leq \gamma n$. Suppose that $G$ is a $(c, h, p)$-typical graph on $n$ vertices. Then, $G$ has a $(K_3,
t)$-design if it is $(K_3,t)$-divisible.
\end{theorem}
We are ready to prove the following lemmas corresponding to the lower bounds in Proposition~\ref{pro::SteineralmostTuran3}.

\begin{lemma}
\label{lem::almoststeiner1}
For every even $t \ge 1$ and sufficiently large $n$, we have
$
\ex(n,\ftp) \geq
\frac{t}{3}\binom{n}{2}-\frac{2t}{3}.
$
\end{lemma}
\begin{proof}
Let $G_1=K_n$, $G_2$ be the graph obtained from $K_n$ by removing an edge and $G_3$ be the graph obtained from $K_n$ by removing a matching containing two edges. Note that $G_1,G_2,G_3$ are all $(c,2812,1)$-typical for every $c\leq 0.9(1/2)^{2812}/(36^{2}\cdot 4^{36})$. We have
$
|E(G_1)|=|E(G_2)|+1=|E(G_3)|+2,
$
and hence there exists $i \in \{1,2,3\}$ such that $3$ divides $|E(G_i)|$. Therefore, $G_i$ is $(K_3,t)$-divisible. By Theorem~\ref{GlockKuhndLoOsthusdesign}, there exists a $(K_3,t)$-design of $G_i$. The corresponding hypergraph $H$,~i.e.,~the $n$-vertex $3$-graph where a triple is a hyperedge in $H$ if it is a $K_3$ in the $(K_3,t)$-design, satisfies that every pair of vertices has codegree at most $t$ and
$
    |H|\ge \frac{t}{3} \left(\b{n}{2} -2\right) =  \frac{t}{3}\binom{n}{2}-\frac{2t}{3}.
$
\end{proof}
\begin{lemma}
\label{lem::almoststeiner2}
For every odd $t\ge 1$ and sufficiently large odd $n$, we have
$
    \ex(n,\ftp)\geq
\frac{t}{3}\binom{n}{2}-\frac{8t}{3}.
$
\end{lemma}
\begin{proof}
Denote by $C_4$ the cycle with $4$ vertices. Let $G_1=K_n$, $G_2$ be the graph obtained from $K_n$ by removing $4$ edges that form a copy of $C_4$ and $G_3$ be the graph obtained from $K_n$ by removing $8$ edges that form two vertex-disjoint copies of $C_4$. Note that $G_1,G_2,G_3$ are all $(c,2812,1)$-typical for every $c\leq 0.9(1/2)^{2812}/(36^{2}\cdot 4^{36})$ and every vertex in each of these three graphs has an even degree. We have
$
|E(G_1)|=|E(G_2)|+4=|E(G_3)|+8,
$
and hence there exists $i\in \{1,2,3\}$ such that $3$ divides $|E(G_i)|$. Therefore, $G_i$ is $(K_3,t)$-divisible. By Theorem~\ref{GlockKuhndLoOsthusdesign}, there exists a $(K_3,t)$-design of $G_i$. The corresponding hypergraph $H$ satisfies that every pair of vertices has codegree at most $t$ and
$
    |H| \ge \frac{t}{3} \left(\b{n}{2} - 8\right) = \frac{t}{3}\binom{n}{2}-\frac{8t}{3}.
$
\end{proof}

\begin{claim}
\label{claim:3-graphsteiner1}
For every even $n \ge 4$, there exists an $n$-vertex $3$-graph $H$ with $d(x,y)\geq 1$ for every $x,y\in V(H)$ and $|H|\leq \frac{1}{3}\binom{n}{2}+\frac{n}{3}$.
\end{claim}
\begin{proof}
Since $n$ is even, by Theorem~\ref{designdehon}, there exists an $(n-1,3,2,1)$-design or an $(n+1,3,2,1)$-design. In the first case, take an $(n-1,3,2,1)$-design on $[n-1]$, and add one vertex $n$ and hyperedges $\{2i-1,2i,n\}$ for $i=1,\ldots ,n/2-1$ and hyperedge $\{1,n-1,n\}$. The resulting hypergraph $H$ has the desired properties: $d(x,y)\ge 1$ for every $x,y\in [n]$ and
\begin{align*}
    |H|= \frac{1}{3}\binom{n-1}{2}+\frac{n}{2}
    =\frac{1}{3}\binom{n}{2}+\frac{n}{6}+\frac{1}{3}
    \leq \frac{1}{3}\binom{n}{2}+\frac{n}{3}.
\end{align*}
Now, suppose there exists an $(n+1,3,2,1)$-design on $[n+1]$. Let $\{v_iu_i(n+1)\}$ be the hyperedges containing vertex $n+1$. 
Remove vertex $n+1$, and then add hyperedge $v_iu_iz_i$, where $z_i$ is an arbitrary vertex, for every $i$. 
The resulting hypergraph $H$ has the desired properties: $d(x,y) \ge 1$ for every $x,y \in [n]$ and
\bdm
    |H| \le \frac{1}{3}\binom{n+1}{2}=\frac{1}{3}\binom{n}{2}+\frac{n}{3}. \qedhere
\edm
\end{proof}
\begin{claim}
\label{claim:3-graphsteiner2}
For every odd $t\geq 1$, there exists a $(t+3)$-vertex $3$-graph $H$ with $|H|\geq \frac{t}{3}\binom{t+3}{2}-\frac{t+3}{3}$ such that each pair of vertices has codegree at most $t$.
\end{claim}
\begin{proof}
Take a copy of the complete $3$-graph on $t+3$ vertices and remove the hyperedge set of a $(t+3)$-vertex $3$-graph $H'$, where $|H'|\le \frac{1}{3}\binom{t+3}{2}+\frac{t+3}{3}$ and $d(x,y) \geq 1$ for every pair of vertices $x,y$ in $H'$. Such a hypergraph exists by Claim~\ref{claim:3-graphsteiner1}. The resulting hypergraph $H$ has the desired properties: every pair of vertices has codegree at most $t$ and
\begin{equation*}
    |H|\ge \binom{t+3}{3}-\left(\frac{1}{3}\binom{t+3}{2}+\frac{t+3}{3}\right)= \frac{t}{3}\binom{t+3}{2}-\frac{t+3}{3}.
    \qedhere
\end{equation*}
\end{proof}

\begin{lemma}
\label{lem::almoststeiner3}
For odd $t \ge 1$ and sufficiently large even $n$, we have
$
    \ex(n,\ftp)\geq
\frac{t}{3}\binom{n}{2}-\frac{n}{3}-\frac{t^2}{6} - 3t.
$
\end{lemma}
\begin{proof}
Let $s= s(n):=\left\lfloor \frac{n}{t+3} \right\rfloor$. Partition $[n]$ into disjoint sets $B_1\cup\ldots \cup B_s \cup J$, where $|B_k|=t+3$ for each $k\in [s]$ and $0 \le |J| < t+3$. Note that $|B_k|$ and $|J|$ are both even. Assume $J=\{w_1,\ldots, w_{|J|}\}$. Define $G_1$ to be the $n$-vertex graph where $uv\not\in E(G_1)$ iff $u,v\in B_k$ for some $k\in[s]$ or $\{u,v\}=\{w_{2j-1}, w_{2j}\}$ for some $j\in [|J|/2]$. Let $G_2$ be the graph obtained from $G_1$ by removing $4$ edges that form a copy of $C_4$, every vertex of which comes from a different set $B_k$. Let $G_3$ be the graph obtained from $G_1$ by removing $8$ edges that form two vertex-disjoint copies of $C_4$, where every vertex in these two $C_4$'s comes from a different set $B_k$. Note that $G_1,G_2,G_3$ are all $(c,2812,1)$-typical for every $c\leq 0.9(1/2)^{2812}/(36^{2}\cdot 4^{36})$ and every vertex in the graphs $G_1, G_2$ and $G_3$ has an even degree. Further, we have that
$
|E(G_1)|=|E(G_2)|+4=|E(G_3)|+8,
$
and hence there exists $i\in\{1,2,3\}$ such that 3 divides $|E(G_i)|$. Therefore, $G_i$ is $(K_3,t)$-divisible. By Theorem~\ref{GlockKuhndLoOsthusdesign}, there exists a $(K_3,t)$-design of $G_i$. Let $H$ be the corresponding $3$-graph. Note that if $u,v\in B_k$ for some $k\in[s]$, then $uv \notin E(G_i)$, and hence $d(u,v) = 0$ in $H$. We have 
\bdm
|H| \ge \frac{t}{3} \left(\b{n}{2} - s\b{t+3}{2} - \frac{|J|}{2} - 8\right).
\edm

Now, we construct an $n$-vertex hypergraph $H'$ from $H$ by adding to every set $B_k$ a copy of a $(t+3)$-vertex $3$-graph with at least $\frac{t}{3}\binom{t+3}{2}-\frac{t+3}{3}$ hyperedges, each pair of vertices in which has codegree at most $t$. Such a hypergraph exists by Claim~\ref{claim:3-graphsteiner2}. Then, $H'$ satisfies that every pair of vertices has codegree at most $t$ and
\begin{align*}
    |H'| &\ge \frac{t}{3} \left(\b{n}{2} - s\b{t+3}{2} - \frac{|J|}{2} - 8\right) 
    +s \left(\frac{t}{3} \b{t+3}{2} - \frac{t+3}{3}\right) \\
    &=  \frac{t}{3}\binom{n}{2} -\frac{t|J|}{6}-\frac{8t}{3}-\frac{s(t+3)}{3}
    \geq \frac{t}{3}\binom{n}{2}-\frac{t^2}{6} - 3t-\frac{n}{3}.
    \qedhere
\end{align*}
\end{proof}

\end{document}